\documentclass[12pt]{article}
\usepackage{amsthm,amssymb,amsmath,txfonts,textcomp,mathrsfs,color}
\usepackage{mathrsfs}
\usepackage{amsmath}
\usepackage{amsfonts}
\usepackage{amssymb}
\usepackage{amsmath,amssymb,amsthm}
\usepackage{amsmath,amssymb,amsthm,amscd}
\usepackage[frame,cmtip,arrow,matrix,line,graph,curve]{xy}
\usepackage{graphpap, color}
\usepackage[mathscr]{eucal}
\usepackage{color}
\usepackage{verbatim}
\usepackage{hyperref}
\usepackage{cite}
\usepackage{slashbox,multirow}
\usepackage{float}
\newtheoremstyle{mythm}{3pt}{3pt}{}{}{\bfseries}{}{5mm}{}

\usepackage{float}
\catcode`!=11
\let\!iint\iint \def\iint{\displaystyle\!iint}
\let\!int\int   \def\int{\displaystyle\!int}
\let\!lim\lim   \def\lim{\displaystyle\!lim}
\let\!sum\sum   \def\sum{\displaystyle\!sum}
\let\!sup\sup   \def\sup{\displaystyle\!sup}
\let\!inf\inf   \def\inf{\displaystyle\!inf}
\let\!cap\cap   \def\cap{\displaystyle\!cap}
\let\!max\max   \def\max{\displaystyle\!max}
\catcode`!=12

\newtheorem{thm}{Theorem}[section]
\newtheorem{lem}[thm]{Lemma}
\newtheorem{exa}{Example}[section]
\newtheorem{defn}[{thm}]{Definition}

\newtheorem{prop}[thm]{Proposition}
\newtheorem{corr}[thm]{Corollary}
\newtheorem{con}[thm]{Conjecture}
\newtheorem*{ack}{Acknowledgment}
\newtheorem{rem}[thm]{Remark}

\allowdisplaybreaks

\numberwithin{equation}{section}
\begin{document}

\title{The  Gaps  of  Consecutive Eigenvalues \\ of Laplacian on  Riemannian Manifolds}
\author{Lingzhong Zeng
\\ \small College of Mathematics and Informational Science,
Jiangxi Normal University,
\\
\small Nanchang 330022, China, E-Mail: lingzhongzeng@yeah.net
}
\date{}

\maketitle

\begin{abstract}\noindent In this paper, we investigate the Dirichlet problem of Laplacian on complete Riemannian manifolds.  By constructing new trial functions,
we obtain a sharp upper bound of the gap of the consecutive eigenvalues  in the sense of the order, which  affirmatively answers to a
conjecture proposed by Chen-Zheng-Yang. In addition, we also exploit the closed eigenvalue problem of
Laplacian and obtain a similar optimal upper bound. As some important examples,  we investigate the eigenvalues of the eigenvalue problem of the Laplacian
 on  the unit sphere and cylinder,
compact homogeneous Riemannian manifolds without boundary, connected bounded domain  and   compact complex hypersurface without boundary in the standard
complex projective space $\mathbb{C}P^{n}(4)$ with holomorphic sectional curvature $4$,
and some intrinsic estimates for the eigenvalue gap is obtained. As the author know, for the Dirichlet problem, the gap  $\lambda_{k+1}-\lambda_{k}$ is
bounded by the first $k$-th eigenvalues in the previous literatures. However, by a large number of numerical calculations,
the author surprisingly find that the gap of the consecutive eigenvalues of the Dirichlet problem on the $n$-dimensional Euclidean space $R^{n}$
can be bounded only by the first two eigenvalues. Therefore, we venture to conjecture that all of the eigenvalues  satisfy:
$\lambda_{k+1}-\lambda_{k}\leq \mathfrak{S}_{i}(\Omega)(\lambda_{2}-\lambda_{1})k^{1/n}$, where $\mathfrak{S}_{i}(\Omega),i=1,2$ denote the first shape coefficient and the second shape coefficient.  In particular, if we consider the second shape coefficient, then there is a
close connection between this universal inequality and the famous Panye-P\'{o}lya-Weinberger conjecture in general form. By calculating  some important examples, we adduce some good
evidence on the correctness of this conjecture.
\vskip3mm
\noindent {\it\bfseries Keywords}:
sharp gap; Laplacian;
consecutive eigenvalue; Riemannian manifold.

\vskip3mm
\noindent 2000 MSC 35P15, 53C40.
\end{abstract}

\section{Introduction}

Let
$\Omega$ be a bounded domain with piecewise smooth
boundary $\partial\Omega$ in an $n$-dimensional complete Riemannian manifold $M^{n}$ and $\Delta$ be the Laplacian on $M^{n}$. We consider the following Dirichlet problem:

\begin{equation}
\left\{ \begin{aligned} \label{Eigenvalue-Problem}\Delta u=-\lambda u,\ \ &{\rm
in}\ \ \ \ \Omega,
         \\
u=0,\ \ &{\rm on}\ \ \partial\Omega,
                          \end{aligned} \right.
                          \end{equation} which is also called a membrane problem \eqref{Eigenvalue-Problem}. This eigenvalue problem has nontrivial
solutions $u$ only for a discrete set of eigenvalues
$\{
\lambda_{k}
\}
_{k\in\mathbb{Z}^{+}}.$ The corresponding eigenfunctions
$\{u_{k}\}_{k\in\mathbb{Z}^{+}}$ provide an orthonormal  basis of $L^{2}(\Omega)$. We may enumerate
the eigenvalues in increasing order
as follows:

\begin{equation*}
0<\lambda_{1}<\lambda_{2}\leq\lambda_{3}\leq\cdots\leq\lambda_{k}\leq\cdots\rightarrow+\infty,
\end{equation*}
where each eigenvalue is repeated according to its multiplicity.
When $M^{n}$ is an $n$-dimensional Euclidean space $\mathbb{R}^{n}$, Payne,
P\'{o}lya and Weinberger \cite{PPW1} and \cite{PPW2} exploited the eigenvalues of the Dirichlet problem \eqref{Eigenvalue-Problem} and obtained the following universal inequality:

\begin{equation}\label{ppw-ineq}\lambda_{k+1}-\lambda_{k}\leq\frac{4}{nk}\sum^{k}_{i=1}\lambda_{i}.\end{equation}
In fact, the universal inequality of Payne, P\'{o}lya and Weinberger is extended  by many mathematicians in some differential backgrounds. However,
there are two main contributions due to Hile and Protter \cite{HP} and
Yang \cite{Y}. In 1980, Hile and Protter \cite{HP} obtained the following universal inequality:

\begin{equation}\label{hp-ineq}\sum^{k}_{i=1}\frac{\lambda_{i}}{\lambda_{k+1}-\lambda_{i}}\geq\frac{nk}{4},\end{equation}
which is sharper than \eqref{ppw-ineq}. Furthermore, Yang \cite{Y} (cf. \cite{CY1}) obtained a very sharp universal inequality of eigenvalues:

\begin{equation}\label{y1-ineq}\sum^{k}_{i=1}(\lambda_{k+1}-\lambda_{i})^{2}\leq\frac{4}{n}\sum^{k}_{i=1}(\lambda_{k+1}-\lambda_{i})\lambda_{i}.\end{equation}
From the inequality \eqref{y1-ineq}, one can yield

\begin{equation}\label{y2-ineq}\lambda_{k+1}\leq\frac{1}{k}\left(1+\frac{4}{n}\right)\sum^{k}_{i=1}\lambda_{i}.\end{equation}
The inequalities \eqref{y1-ineq} and \eqref{y2-ineq} are called by Ashbaugh Yang's first inequality and second inequality,
respectively (cf. \cite{Ash8}, \cite{Ash9}). Also we note that Ashbaugh and Benguria gave an optimal
estimate for $k = 1$, see \cite{ Ash2,Ash3,Ash4}. By utilizing the Chebyshev's inequality, it is not difficult
to prove the following inequalities (cf. \cite{Ash9}):

$$\eqref{y1-ineq}\Rightarrow \eqref{y2-ineq} \Rightarrow \eqref{hp-ineq} \Rightarrow \eqref{ppw-ineq}.$$
In 2007, Cheng and Yang established a celebrated recursion formula \cite{CY3}.  By making use of this recursion formula, they obtained a sharp upper bound
of the $(k+1)$-th eigenvalue, this is, they proved the following:

\begin{equation}
\begin{aligned}
\label{Cheng-Yang-ineq}\lambda_{k+1}
\leq C_{0}(n,k)k^{\frac{2}{n}}\lambda_{1},
\end{aligned}
\end{equation}
where $C_{0}(n,k)\leq1+\frac{4}{n}$ is a constant depending on $n$ and $k$ (see Cheng and Yang's paper \cite{CY3}). From the Weyl's
asymptotic formula(cf. \cite{Wey,Weyl1,Bera}):
\begin{equation}
\begin{aligned}
\label{Weyl-asmptotic-formula}\lambda_{k}\sim\frac{4\pi^{2}}{(\omega_{n}V(\Omega))^{\frac{2}{n}}}k^{\frac{2}{n}}\ \ as \ \ k\rightarrow+\infty,\end{aligned}
\end{equation}
where $\omega_{n}$ and $V(\Omega)$ denote the volumes of the $n$-dimensional unit ball $\mathbb{B}^{n}(1)\subset\mathbb{R}^{n}$ and the bounded domain $\Omega$, we know that the upper bound \eqref{Cheng-Yang-ineq} of Cheng and Yang is
best possible in the meaning of the order on $k$. If $M^{n}$ is a complete Riemannian manifold isometrically
immersed in a Euclidean space  $\mathbb{R}^{n+p}$, Chen-Cheng \cite{CC} derived an extrinsic estimates as follows:

\begin{equation}\label{cc1-ineq}\sum^{k}_{i=1}(\lambda_{k+1}-\lambda_{i})^{2}\leq\frac{4}{n}\sum^{k}_{i=1}(\lambda_{k+1}-\lambda_{i})\left(\lambda_{i}+\frac{n^{2}\|H\|^{2}}{4}\right).\end{equation}
$H$ is the mean curvature vector field of $M^{n}$ with $\|H\|^{2}=\sup_{\Omega}|H|^{2}.$ Further, by using the Cheng-Yang's recursive formula in \cite{CY1}, they also obtained an upper bound estimates, this is,

\begin{equation}\label{cc2-ineq}\left(\lambda_{k+1}+\frac{n^{2}\|H\|^{2}}{4}\right)\leq C_{0}(n)k^{\frac{2}{n}}\left(\lambda_{1}+\frac{n^{2}\|H\|^{2}}{4}\right).\end{equation}
Suppose that $\Omega$ is a bounded connected domain in a unit sphere $\mathbb{S}^{n}(1)$,  Cheng and Yang \cite{CY1} obtained an upper bound estimate for the gap of the consecutive eigenvalues of the eigenvalue problem \eqref{Eigenvalue-Problem}:

\begin{equation*}\begin{aligned}\lambda_{k+1}-\lambda_{k}\leq2\left[\left(
\frac{2}{n}\frac{1}{k}\sum^{k}_{i=1}\lambda_{i}+\frac{n}{2}\right)^{2}-\left(1+\frac{4}{n}\right)\frac{1}{k}\sum^{k}_{j=1}\left(\lambda_{j}
-\frac{1}{k}\sum^{k}_{i=1}\lambda_{i}\right)^{2}\right]^{\frac{1}{2}}.\end{aligned}
\end{equation*} In \cite{CY2}, Cheng and Yang
investigated the Dirichlet  problem \eqref{Eigenvalue-Problem} of
the Laplacian on a connected bounded domain of the standard complex projective space
$\mathbb{C}P^{n}(4)$ with holomorphic sectional curvature $4$. They proved the following universal inequality
\begin{equation}\label{cy-universal-1}\begin{aligned}\sum^{k}_{i=1}\left(\lambda_{k+1}-\lambda_{i}\right)^{2}\leq\frac{2}{n}
\sum^{k}_{i=1}\left(\lambda_{k+1}-\lambda_{i}\right)\left(\lambda_{i}+2n(n+1)\right).\end{aligned}
\end{equation}
By \eqref{cy-universal-1}, Cheng and Yang gave an explicit
estimate of the $k+1$-th eigenvalue of Laplacian on such objects by its first $k$
eigenvalues and proved the following inequality (cf. \cite{CY2}):

\begin{equation}\label{cylk1-lk}\begin{aligned}\lambda_{k+1}-\lambda_{k}\leq
2\left\{\left[\frac{1}{n}\frac{1}{k}\sum^{k}_{i=1}\lambda_{i}+2(n+1)\right]^{2}
-\left(1+\frac{2}{n}\right)\frac{1}{k}\sum^{k}_{j=1}\left(\lambda_{j}-\frac{1}{k}\sum^{k}_{i=1}\lambda_{i}\right)^{2}\right\}^{\frac{1}{2}}.\end{aligned}
\end{equation}Let $\Omega$ be a bounded domain on an $n$-dimensional Euclidean space $\mathbb{R}^{n}$, Chen, Zheng and Yang \cite{CZY} recently established a gap of consecutive
eigenvalues of the eigenvalue problem \eqref{Eigenvalue-Problem},

\begin{equation}\label{czy-1}\begin{aligned}\lambda_{k+1}-\lambda_{k}\leq
C_{n,\Omega}k^{\frac{1}{n}},\end{aligned}
\end{equation}
where $$C_{n,\Omega}=4\lambda_{1}\sqrt{\frac{C_{0}(n)}{n}},$$ and the constant $C_{0}(n)$ is the same as the one in \eqref{Cheng-Yang-ineq}. By a direct calculation and using the Weyl's asymptotic formula, we know the order of the upper
bound of the gap of the consecutive eigenvalues of $\mathbb{S}^{n}$ with
standard metric is $k^{\frac{1}{n}}$. Therefore, for general Dirichlet
problem of the Laplacian on Riemannian manifolds, Chen, Zheng and Yang presented a conjecture as follows:

\begin{con}\label{con1}(cf. {\rm \cite{CZY}})
Let $(M^{n},g)$ be an $n$-dimensional Riemannian manifold, and $\lambda_{i}$ be the $i$-th $(i=1,2,\cdots,k)$
eigenvalue of the eigenvalue problem \eqref{Eigenvalue-Problem}. Then we have

\begin{equation}\label{czy-conj}\begin{aligned}\lambda_{k+1}-\lambda_{k}\leq
C_{n,\Omega}k^{\frac{1}{n}},\end{aligned}
\end{equation}
where $C_{n,\Omega}$ is a constant dependent on $\Omega$ itself and the dimension $n$.
\end{con}
Furthermore, as an excellent example to support conjecture \ref{con1}, Chen, Zheng and Yang also investigated the eigenvalues of Laplacian on hyperbolic space.
They proved the eigenvalue inequality \eqref{czy-conj} also holds for some  hyperbolic space with some curvature conditions.

In addition, the famous fundamental gap conjecture states that, for the Dirichlet eigenvalue
problem of the Schr\"{o}dinger operator,

\begin{equation}
\left\{ \begin{aligned}\label{gap-conj}\Delta u-V(x) u=-\lambda u,\ \ &{\rm
in}\ \ \ \ \Omega,
         \\
u=0,\ \ &{\rm on}\ \ \partial\Omega,
                          \end{aligned} \right.
                          \end{equation}where $V(x)$ is a convex potential, Then the eigenvalues of \eqref{gap-conj} satisfy $\lambda_{2}-\lambda_{1}\geq3\pi^{2}/D^{2}.$
Many mathematicians have contributed much to this conjecture  (cf. {\rm \cite{AC,B,SWYY,YY,YZ}} and the references therein), and it was finally solved by B. Andrews and J. Clutterbuck in {\rm  \cite{AC}}.

In this paper, we exploit the Dirichlet problem \eqref{Eigenvalue-Problem} of the Laplacian on the complete Riemannian manifolds.
Suppose that $M^{n}$ is an $n$-dimensional complete Riemannian manifolds isometrically
immersed in a Euclidean space. For this case, we answer to the conjecture \ref{con1}. This is to say, we prove the
following:

\begin{thm}
\label{thm1.1}Let $(M^{n},g)$ be  an $n$-dimensional complete Riemannian manifolds isometrically
immersed in a Euclidean space $\mathbb{R}^{n+p}$, and $\lambda_{i}$ be the $i$-th $(i=1,2,\cdots,k)$
eigenvalue of the Dirichlet problem \eqref{Eigenvalue-Problem}. Then we have

\begin{equation}\label{z1}\begin{aligned}\lambda_{k+1}-\lambda_{k}\leq
C_{n,\Omega}k^{\frac{1}{n}},\end{aligned}
\end{equation}
where $C_{n,\Omega}$ is a constant dependent on $\Omega$ itself and the dimension $n$.
\end{thm}
Next, assume that $M^{n}$ is an $n$-dimensional closed Riemannian manifold. We also consider the closed eigenvalue problem of Laplacian:

\begin{equation}
 \label{Eigen-Prob-closed}\Delta u=-\overline{\lambda} u,\ \ {\rm
in}\ \ \ \ M^{n}.\end{equation}
It is well known that the spectrum of the closed eigenvalue problem  \eqref{Eigen-Prob-closed}  is
discrete and satisfies the following:

\begin{equation*}
0=\overline{\lambda}_{0}<\overline{\lambda}_{1}\leq\overline{\lambda}_{2}\leq\overline{\lambda}_{3}\leq\cdots\leq\overline{\lambda}_{k}\leq\cdots\rightarrow+\infty,
\end{equation*}
where $\overline{\lambda}_{k}$ is the $k$-th eigenvalue of the closed eigenvalue problem \eqref{Eigen-Prob-closed} and each eigenvalue is repeated according to its multiplicity.

When $M^{n}$ is an $n$-dimensional compact homogeneous Riemannian manifold, for the closed eigenvalue problem \eqref{Eigen-Prob-closed}, Li\cite{L} proved

\begin{equation*}\begin{aligned}\overline{\lambda}_{k+1}-\overline{\lambda}_{k}\leq
\frac{2}{k+1}\left(\sqrt{\left(\sum^{k}_{i=1}\overline{\lambda}_{i}\right)^{2}+(k+1)\sum^{k}_{i=1}\overline{\lambda}_{i}\overline{\lambda}_{1}}
+\sum^{k}_{i=1}\overline{\lambda}_{i}\right)+\overline{\lambda}_{1},\end{aligned}
\end{equation*}
 When $M^{n}$ is
an $n$-dimensional compact minimal submanifold in a unit sphere $S^{N}(1)$,  P. C.Yang and Yau \cite{YY} proved the eigenvalues of the closed eigenvalue problem satisfy the following eigenvalue inequality:
\begin{equation*}\begin{aligned}\overline{\lambda}_{k+1}-\overline{\lambda}_{k}\leq n+
\frac{2}{n(k+1)}\left(\sqrt{\left(\sum^{k}_{i=1}\overline{\lambda}_{i}\right)^{2}+n^{2}(k+1)
\sum^{k}_{i=1}\overline{\lambda}_{i}\overline{\lambda}_{1}}+\sum^{k}_{i=1}\overline{\lambda}_{i}\right).\end{aligned}\end{equation*} Furthermore, Harrel II and Michel (see \cite{HM} and \cite{H3}) obtained an abstract inequality of algebraic version. By applying the algebraic inequality, they proved that,
if $M^{n}$ is an $n$-dimensional compact minimal submanifold in a unit
sphere $\mathbb{S}^{N}(1)$, then

\begin{equation}\begin{aligned}\label{hm-ineq-1}\overline{\lambda}_{k+1}-\overline{\lambda}_{k}\leq n+
\frac{4}{n(k+1)}\sum^{k}_{i=1}\overline{\lambda}_{i},\end{aligned}
\end{equation}
and if $M^{n}$ is an $n$-dimensional compact homogeneous Riemannian manifold, then

\begin{equation}\begin{aligned}\label{hm-ineq-2}\overline{\lambda}_{k+1}-\overline{\lambda}_{k}\leq
\frac{4}{k+1}\sum^{k}_{i=1}\overline{\lambda}_{i}+\overline{\lambda}_{1},\end{aligned}
\end{equation} Therefore, the above inequalities \eqref{hm-ineq-1} and \eqref{hm-ineq-2} made significant improvement to earlier
estimates of differences of consecutive eigenvalues of Laplacian introduced by
P. C. Yang and Yau  \cite{YY}, Leung  \cite{Le}, Li \cite{L} and Harrel II \cite{HM}.
Cheng and Yang \cite{CY1} also considered the same eigenvalue problem and proved that, when $M^{n}$
is an $n$-dimensional compact homogeneous Riemannian manifold  without  boundary, then the eigenvalues of the
close eigenvalue problem \eqref{Eigen-Prob-closed} satisfy

\begin{equation*}\begin{aligned}\overline{\lambda}_{k+1}-\overline{\lambda}_{k}\leq\left[\left(
\frac{4}{k+1}\sum^{k}_{i=1}\overline{\lambda}_{i}+\overline{\lambda}_{1}\right)^{2}-\frac{20}{k+1}\sum^{k}_{i=0}\left(\overline{\lambda}_{i}
-\frac{1}{k+1}\sum^{k}_{j=1}\overline{\lambda}_{j}\right)^{2}\right]^{\frac{1}{2}};\end{aligned}
\end{equation*}
and when $M^{n}$ is an $n$-dimensional compact minimal submanifold without boundary in a unit sphere $\mathbb{S}^{N}(1)$, then the eigenvalues of the
close eigenvalue problem \eqref{Eigen-Prob-closed} satisfy

\begin{equation*}\begin{aligned}\overline{\lambda}_{k+1}-\overline{\lambda}_{k}\leq2\left[\left(
\frac{2}{n}\frac{1}{k}\sum^{k}_{i=0}\overline{\lambda}_{i}+\frac{n}{2}\right)^{2}-\left(1+\frac{4}{n}\right)\frac{1}{k+1}\sum^{k}_{j=0}\left(\overline{\lambda}_{j}
-\frac{1}{k}\sum^{k}_{i=0}\overline{\lambda}_{i}\right)^{2}\right]^{\frac{1}{2}}.\end{aligned}
\end{equation*}
In \cite{CY2}, Cheng and Yang
investigated the closed eigenvalue problem \eqref{Eigen-Prob-closed} of the Laplacian on a
compact complex hypersurface without boundary in the standard complex projective space $\mathbb{C}P^{n}(4)$  with holomorphic sectional curvature $4$.
They proved the following universal inequality
\begin{equation}\label{cy-universal-2}\begin{aligned}\sum^{k}_{i=0}\left(\overline{\lambda}_{k+1}-\overline{\lambda}_{i}\right)^{2}\leq\frac{2}{n}
\sum^{k}_{i=1}\left(\overline{\lambda}_{k+1}-\overline{\lambda}_{i}\right)\left(\overline{\lambda}_{i}+2n(n+1)\right).\end{aligned}
\end{equation}
By \eqref{cy-universal-2}, Cheng and Yang gave an explicit
estimate of the $k+1$-th eigenvalue of Laplacian on such objects by its first $k$
eigenvalues and proved the following inequality (cf. \cite{CY2}):

\begin{equation}\label{cylk1-lk-2}\begin{aligned}\overline{\lambda}_{k+1}-\overline{\lambda}_{k}\leq
2\left\{\left[\frac{1}{n}\frac{1}{k+1}\sum^{k}_{i=1}\overline{\lambda}_{i}+2(n+1)\right]^{2}
-\left(1+\frac{2}{n}\right)\frac{1}{k+1}\sum^{k}_{j=1}\left(\overline{\lambda}_{j}-\frac{1}{k+1}\sum^{k}_{i=1}\overline{\lambda}_{i}\right)^{2}\right\}^{\frac{1}{2}}.\end{aligned}
\end{equation}

In this paper, we investigate the eigenvalues of the closed eigenvalue problem \eqref{Eigen-Prob-closed} of the Laplacian on the closed Riemannian manifolds and prove the following:

\begin{thm}
\label{thm1.2}Let $(M^{n},g)$ be an $n$-dimensional closed Riemannian manifold, which is isometrically immersed into $(n+\overline{p})$-dimensional Euclidean space $\mathbb{R}^{n+\overline{p}}$, and $\overline{\lambda}_{i}$ be the $i$-th $(i=0,1,2,\cdots)$
eigenvalue of the closed eigenvalue problem \eqref{Eigen-Prob-closed}. Then, for any $k\geq1$, we have

\begin{equation}\label{z1}\begin{aligned}\overline{\lambda}_{k+1}-\overline{\lambda}_{k}\leq
\overline{C}_{n,M^{n}}k^{\frac{1}{n}},\end{aligned}
\end{equation}
where $\overline{C}_{n,M^{n}}$ is a constant dependent on $M^{n}$ itself and the dimension $n$.
\end{thm}

\noindent
This paper is organized as follows. In section
\ref{sec2}, we prove several key lemmas and establish several general formulas of the eigenvalues of the Dirichlet problem. In addition,by the same method, we also yield the corresponding general formulas with respect to the closed eigenvalue problem. By
utilizing those general formulas, we give the proofs of
theorem \ref{thm1.1} and theorem \ref{thm1.2} in section \ref{sec3}. We exploit the eigenvalue of the Dirichlet problem on the unit sphere and cylinder in section \ref{sec4}.
In section \ref{sec5}, we investigate the eigenvalues of the eigenvalue problem of the Laplacian
on a connected bounded domain  and on a compact complex hypersurface without boundary in the standard
complex projective space $\mathbb{C}P^{n}(4)$ with holomorphic sectional curvature $4$. In addition, we consider the eigenvalues of the closed eigenvalue problem of Laplacian
on the compact Riemannian manifolds without boundary in section \ref{sec6}. In the last section, we give some gap conjectures of consecutive eigenvalues of the Dirichlet problem on complete Riemannian manifolds. As a further interest,
we provide some important examples to support those conjectures proposed in this section.

\vskip5mm

\section{Some Technical Lemmas}\label{sec2}
\vskip3mm

In order to give the proofs of theorem \ref{thm1.1} and theorem \ref{thm1.2}, we would like to prove some key lemmas in this section. At first, we recall
the following algebraic inequality which is proved by Chen, Zheng and Yang in \cite{CZY}. By applying this algebraic inequality, Chen, Zheng and Yang established
the following general formula (see lemma 2.1 in \cite{CZY}).

\begin{lem}\label{lem2.1}
Let $(M^{n},g)$ be an $n$-dimensional complete Riemannian manifold and $\Omega$ a bounded domain
with piecewise smooth boundary $\partial\Omega$ on $M^{n}$.
Assume that $\lambda_{i}$  is  the $i^{\text{th}}$ eigenvalue of the
Dirichlet problem \eqref{Eigenvalue-Problem}  and $u_{i}$ is an
orthonormal eigenfunction corresponding to $\lambda_{i}$, $i =
1,2,\cdots$, such that

\begin{equation*}
\left\{ \begin{aligned}\Delta u_{i}=-\lambda u_{i},\ \ \ \ \ \ \ &
in\ \ \ \ \Omega,
         \\ u_{i}=0,\ \ \ \ \ \ \ \ \ \ \ \ \ \ &  on\ \ \partial\Omega,
\\
\int_{\Omega}
u_{i}u_{j}dv=\delta_{ij},\ \ & for~any  \ i,j=1,2,\cdots.
                          \end{aligned} \right.
                          \end{equation*}
Then, for any function $h(x)\in C^{3}(\Omega)\cap C^{2}(\overline{\Omega})$ and any
integer $k,i \in \mathbb{Z}^{+},~(k>i\geq1)$,  eigenvalues of the Dirichlet problem \eqref{Eigenvalue-Problem} satisfy

\begin{equation}\begin{aligned}\label{general-formula-1}&((\lambda_{k+2}-\lambda_{i})+(\lambda_{k+1}-\lambda_{i}))\|\nabla hu_{i}\|^{2}
\\&\leq2\sqrt{(\lambda_{k+2}-\lambda_{i})(\lambda_{k+1}-\lambda_{i})\||\nabla h|^{2}u_{i}\|^{2}}
+\|2\langle\nabla h,\nabla u_{i}\rangle+u_{i}\Delta h\|^{2},
\end{aligned}\end{equation}
where
\begin{equation*}
\|h(x)\|^{2} =\int_{\Omega}h^{2}(x)dv.
\end{equation*}
\end{lem}

For the closed eigenvalue problem, we can also prove the following by the same method given by:

\begin{lem}\label{lem2.2}
Let $(M^{n},g)$ be an $n$-dimensional closed Riemannian manifold.
Assume that $\overline{\lambda}_{i}$  is  the $i^{\text{th}}$ eigenvalue of the
eigenvalue problem \eqref{Eigen-Prob-closed}  and $u_{i}$ is an
orthonormal eigenfunction corresponding to $\overline{\lambda}_{i}$, $i =0,
1,2,\cdots$, such that
\begin{equation*}
\left\{ \begin{aligned}\Delta u_{i}=-\lambda u_{i},\ \ \ \ \ \ \ &
in\ \ \ \ \Omega,
\\
\int_{\Omega}
u_{i}u_{j}dv=\delta_{ij},\ \ & for~any  \ i,j=0,1,2,\cdots.
                          \end{aligned} \right.
                          \end{equation*}
Then, for any function $h(x)\in C^{3}(\Omega)\cap C^{2}(\overline{\Omega})$ and any
integer $k,i \in \mathbb{Z},~(k>i\geq0)$,  eigenvalues of the closed eigenvalue problem \eqref{Eigen-Prob-closed} satisfy

\begin{equation}\begin{aligned}\label{general-formula-1}&((\overline{\lambda}_{k+2}-\overline{\lambda}_{i})+(\overline{\lambda}_{k+1}-\overline{\lambda}_{i}))\|\nabla hu_{i}\|^{2}
\\&\leq2\sqrt{(\overline{\lambda}_{k+2}-\overline{\lambda}_{i})(\overline{\lambda}_{k+1}-\overline{\lambda}_{i})\||\nabla h|^{2}u_{i}\|^{2}}
+\|2\langle\nabla h,\nabla u_{i}\rangle+u_{i}\Delta h\|^{2},
\end{aligned}\end{equation}
where
\begin{equation*}
\|h(x)\|^{2} =\int_{\Omega}h^{2}(x)dv.
\end{equation*}
\end{lem}
\begin{proof}Recall that the proof of lemma \ref{lem2.1} given by Chen-Zheng-Yang in \cite{CZY} is very fascinating
and the key strategy is to apply the Rayleigh-Ritz inequality and Lagrange method of multipliers in real Banach spaces .
By the same strategy as the one in \cite{CZY}, it is not difficult to give the proof of this lemma if one notices to count the number of eigenvalues from $0$.
Here, we omit it.\end{proof}
By applying lemma \ref{lem2.1}, we have

\begin{lem}\label{lem2.3}Let $\rho $ be a constant such that, for any $i=1,2,\cdots,k,$ $\lambda_{i}+\rho>0$.
Under the assumption of the lemma \ref{lem2.1}, for any $j=1,2,\cdots,l,$ and any real value function $h_{j}\in C^{3}(\Omega)\cap C^{2}(\overline{\Omega})$,  we have

\begin{equation}\begin{aligned}\label{general-formula-2}\sum_{j=1}^{l}\frac{a_{j}^{2}+b_{j}}{2}\left(\lambda_{k+2}-\lambda_{k+1}\right)^{2}
\leq4(\lambda_{k+2}+\rho)\sum_{j=1}^{l}\|2\langle\nabla h_{j},\nabla u_{i}\rangle+u_{i}\Delta h_{j}\|^{2}
,
\end{aligned}\end{equation}
where

\begin{equation*}a_{j}=\sqrt{\|\nabla h_{j}u_{i}\|^{2}},\end{equation*}

\begin{equation*}b_{j}=\sqrt{\||\nabla h_{j}|^{2}u_{i}\|^{2}},\end{equation*}
\begin{equation}a_{j}^{2}\geq b_{j},\end{equation}  and \begin{equation*}
\|h(x)\|^{2}=\int_{\Omega}h^{2}(x)dv.
\end{equation*}

\end{lem}

\begin{proof}By the assumption in this lemma, we have

\begin{equation*}\begin{aligned}\frac{a_{j}^{2}-b_{j}}{2}
\left(\sqrt{\lambda_{k+2}-\lambda_{i}}+\sqrt{\lambda_{k+1}-\lambda_{i}}\right)^{2}\geq0,\end{aligned}\end{equation*}
which is equivalent to the following:

\begin{equation}\begin{aligned}\label{gap-ineq}&a_{j}^{2}((\lambda_{k+2}-\lambda_{i})+(\lambda_{k+1}-\lambda_{i}))
-2b_{j}\sqrt{(\lambda_{k+2}-\lambda_{i})(\lambda_{k+1}-\lambda_{i})}\\&\geq\frac{a_{j}^{2}+b_{j}}{2}
\left(\sqrt{\lambda_{k+2}-\lambda_{i}}-\sqrt{\lambda_{k+1}-\lambda_{i}}\right)^{2}.\end{aligned}\end{equation}
By \eqref{gap-ineq} and \eqref{general-formula-1}, we have
\begin{equation*}\begin{aligned}\frac{a_{j}^{2}+b_{j}}{2}
\left(\sqrt{\lambda_{k+2}-\lambda_{i}}-\sqrt{\lambda_{k+1}-\lambda_{i}}\right)^{2}
&\leq\|2\langle\nabla h_{j},\nabla u_{i}\rangle+u_{i}\Delta h_{j}\|^{2}.
\end{aligned}\end{equation*}
Taking sum over $j$ from $1$ to $l$, we yield

\begin{equation}\begin{aligned}\label{gen-for-4}\sum_{j=1}^{l}\frac{a_{j}^{2}+b_{j}}{2}\left(\sqrt{\lambda_{k+2}-\lambda_{i}}-\sqrt{\lambda_{k+1}-\lambda_{i}}\right)^{2} \leq\sum_{j=1}^{l}\|2\langle\nabla h_{j},\nabla u_{i}\rangle+u_{i}\Delta h_{j}\|^{2}.
\end{aligned}\end{equation}
Multiplying \eqref{gen-for-4} by
$\left(\sqrt{\lambda_{k+2}-\lambda_{i}}+\sqrt{\lambda_{k+1}-\lambda_{i}}\right)^{2}$ on both sides, one can infer that

\begin{equation*}\begin{aligned}\sum_{j=1}^{l}\frac{a_{j}^{2}+b_{j}}{2}\left(\lambda_{k+2}-\lambda_{k+1}\right)^{2}
&\leq\sum_{j=1}^{l}\|2\langle\nabla h_{j},\nabla u_{i}\rangle+u_{i}\Delta h_{j}\|^{2}\left(\sqrt{\lambda_{k+2}-\lambda_{i}}+\sqrt{\lambda_{k+1}-\lambda_{i}}\right)^{2}\\
&=\sum_{j=1}^{l}\|2\langle\nabla h_{j},\nabla u_{i}\rangle+u_{i}\Delta h_{j}\|^{2}\\&\times\left(\sqrt{(\lambda_{k+2}+\rho)-(\lambda_{i}+\rho)}+\sqrt{(\lambda_{k+1}+\rho)-(\lambda_{i}+\rho)}\right)^{2}\\
&\leq4(\lambda_{k+2}+\rho)\sum_{j=1}^{l}\|2\langle\nabla h_{j},\nabla u_{i}\rangle+u_{i}\Delta h_{j}\|^{2}
.
\end{aligned}\end{equation*}
which is the inequality \eqref{general-formula-2}.
Therefore, we finish the proof of this lemma.

\end{proof}

\begin{rem}

Recall that, under the assumption that $\|\nabla h\|=1$, by utilizing \eqref{general-formula-1}, Chen, Zheng and Yang \cite{CZY} obtained

\begin{equation}\label{gene-form-3}\left(\lambda_{k+2}-\lambda_{k+1}\right)^{2}
\leq4\lambda_{k+2}\sum_{j=1}^{l}\|2\langle\nabla h_{j},\nabla u_{i}\rangle+u_{i}\Delta h_{j}\|^{2},\end{equation} which plays a significant  role in estimating the gap of $\lambda_{k+1}-\lambda_{k}$.  If there is no this assumption, it will encounter great difficulties of computing or estimating the detailed value of the term $|\|\nabla h\|^{2}u_{i}|$ and thus to obtain \eqref{gene-form-3} even if $h$ is a standard coordinate function on Euclidean space. However, we notice that the assumption that $|\nabla h|=1$ can be replaced by the assumption that the trial function $h$ satisfies the following condition: \begin{equation*}\|\nabla hu_{i}\|^{2}\geq\sqrt{\||\nabla h|^{2}u_{i}\|^{2}}.\end{equation*} Under the assumption, we can obtain the inequality \eqref{general-formula-2}, which plays a significant role in estimating the gap of eigenvalues of Laplacian on general Riemannian manifolds.

\end{rem}

By the same method as the proof of lemma \ref{lem2.3}, we can prove the following lemma if one notices to count the number of eigenvalues from $0$.

\begin{lem}\label{lem2.4}Let $\rho $ be a constant such that, for any $i=0,1,2,\cdots,k,$ $\overline{\lambda}_{i}+\rho>0$.
Under the assumption of the lemma \ref{lem2.2}, for any $j=0,1,2,\cdots,l,$ and any real value function $h_{j}\in C^{2}(M^{n})$,  we have

\begin{equation}\begin{aligned}\label{general-formula-4}\sum_{j=1}^{l}\frac{a_{j}^{2}+b_{j}}{2}\left(\overline{\lambda}_{k+2}-\overline{\lambda}_{k+1}\right)^{2}
\leq4(\overline{\lambda}_{k+2}+\rho)\sum_{j=1}^{l}\|2\langle\nabla h_{j},\nabla u_{i}\rangle+u_{i}\Delta h_{j}\|^{2}
,
\end{aligned}\end{equation}
where

\begin{equation*}a_{j}=\sqrt{\|\nabla h_{j}u_{i}\|^{2}},\end{equation*}

\begin{equation*}b_{j}=\sqrt{\||\nabla h_{j}|^{2}u_{i}\|^{2}},\end{equation*}
\begin{equation}a_{j}^{2}\geq b_{j},\end{equation}  and \begin{equation*}
\|h(x)\|=\int_{\Omega}h(x)dv.
\end{equation*}

\end{lem}

\section{Proofs of theorem \ref{thm1.1} and theorem \ref{thm1.2}}\label{sec3}
\vskip 3mm

In this section, we would like to give the proofs of theorem \ref{thm1.1} and theorem \ref{thm1.2}. Firstly, we need the
following lemma which can be
found in \cite{CC}.

\begin{lem}\label{lem3.1}
For  an $n$-dimensional  submanifold  $M^{n}$ in Euclidean space
$\mathbb{R}^{n+p}$,   let $y=(y^{1},y^{2},\cdots,y^{n+p})$ is the
position vector of a point $p\in M^{n}$ with
$y^{\alpha}=y^{\alpha}(x_{1}, \cdots, x_{n})$, $1\leq \alpha\leq
n+p$, where $(x_{1}, \cdots, x_{n})$ denotes a local coordinate
system of $M^n$. Then, we have
\begin{equation*}
\sum^{n+p}_{\alpha=1}g(\nabla y^{\alpha},\nabla y^{\alpha})= n,
\end{equation*}
\begin{equation*}
\begin{aligned}
\sum^{n+p}_{\alpha=1}g(\nabla y^{\alpha},\nabla u)g(\nabla
y^{\alpha},\nabla w)=g(\nabla u,\nabla w),
\end{aligned}
\end{equation*}
for any functions  $u, w\in C^{1}(M^{n})$,
\begin{equation*}
\begin{aligned}
\sum^{n+p}_{\alpha=1}(\Delta y^{\alpha})^{2}=n^{2}H^{2},
\end{aligned}
\end{equation*}
\begin{equation*}
\begin{aligned}
\sum^{n+p}_{\alpha=1}\Delta y^{\alpha}\nabla y^{\alpha}= 0,
\end{aligned}
\end{equation*}
where $H$ is the mean curvature of $M^{n}$.
\end{lem}

\vskip 2mm

\textbf{\emph{Proof of theorem}} \ref{thm1.1}. Let $a_{1},a_{2},\cdots,a_{n+p}$ are $(n+p)$ positive number. We define $(n+p)$
scarling coordinate functions $h_{j}(x)=\alpha_{j}x^{j}$, such that

\begin{equation}\label{a-b-in}a_{j}^{2}=\|\nabla h_{j}u_{i}\|^{2}\geq\sqrt{\||\nabla h_{j}|^{2}u_{i}\|^{2}}=b_{j}\geq0,\end{equation} and
\begin{equation}\begin{aligned} \label{sum-3.2}\sum_{j=1}^{n+p}\int2 u_{i}\langle\nabla h_{j},\nabla u_{i}\rangle\Delta  h_{j}dv&=0,\end{aligned}\end{equation}where   $j=1,2,\cdots,n+p,$ and $x^{j}$ denotes the $j$-th standard coordinate function of the Euclidean space $\mathbb{R}^{n+p}$.
Let $$\alpha=\min_{1\leq j\leq n+p}\{\alpha_{j}\},$$$$\overline{\alpha}=\max_{1\leq j\leq n+p}\{\alpha_{j}\},$$$$\beta=\min_{1\leq j\leq n+p}\{b_{j}\},$$
and $l=n+p$, then, by lemma \ref{lem2.1}, we have

\begin{equation}\begin{aligned}\label{left-1}\sum_{j=1}^{l}\frac{a_{j}^{2}+b_{j}}{2}&=\sum_{j=1}^{n+p}\frac{a_{j}^{2}+b_{j}}{2}\\&\geq\frac{1}{2}\left(n\alpha^{2}+
\sum_{j=1}^{n+p}b_{j}\right)\\&\geq\frac{1}{2}\left(n\alpha^{2}+(n+p)\beta\right),\end{aligned}\end{equation}

\begin{equation}\label{nH}\sum^{n+p}_{j=1}(\Delta h_{j})^{2}\leq\overline{\alpha}^{2}n^{2}H^{2},\end{equation} and

\begin{equation}\label{h-u}\sum_{j=1}^{n+p}\int_{\Omega}\langle\nabla h_{j},\nabla u_{i}\rangle^{2}dv\leq\overline{\alpha}^{2}\sum_{j=1}^{n+p}\int_{\Omega}\langle\nabla x^{j},\nabla u_{i}\rangle^{2}dv=\overline{\alpha}^{2}\lambda_{i}.\end{equation}
Since eigenvalues are invariant under isometries, defining
$$
c=\frac{1}{4}\inf_{\psi\in \Psi}\max_{\Omega}n^{2}H^{2}>0,
$$
where $\Psi$ denotes the set of all isometric immersions from $M^n$
into a Euclidean space, by lemma \ref{lem3.1}, \eqref{sum-3.2}, \eqref{nH}, and \eqref{h-u}, we have

\begin{equation}\begin{aligned}\label{right-1}4(\lambda_{k+2}+c)\sum_{j=1}^{n+p}\|2\langle\nabla h_{j},\nabla u_{i}\rangle
+u_{i}\Delta h_{j}\|^{2}&\leq4(\lambda_{k+2}+c)\overline{\alpha}^{2}\left(4\lambda_{i}+\int_{\Omega}u^{2}_{i}n^{2}H^{2}dv\right)
\\&\leq16\lambda_{k+2}\overline{\alpha}^{2}\left(\lambda_{i}+c\right).\end{aligned}\end{equation}
Let $i=1,\rho=c$, then, substituting  \eqref{left-1} and \eqref{right-1} into \eqref{general-formula-2}, we
have

\begin{equation}\begin{aligned}\label{gap1}\left(\lambda_{k+2}-\lambda_{k+1}\right)^{2}
\leq\frac{32\overline{\alpha}^{2}(\lambda_{k+2}+c)}{n\alpha^{2}+(n+p)\beta}\left(\lambda_{1}+c\right)
,
\end{aligned}\end{equation}
Therefore, we deduce from \eqref{gap1} that,
\begin{equation*}\begin{aligned}\lambda_{k+2}-\lambda_{k+1}&\leq
\sqrt{\frac{32\overline{\alpha}^{2}}{n\alpha^{2}+(n+p)\beta}}\sqrt{\lambda_{1}+c}\sqrt{\lambda_{k+2}+c}\\&\leq
(\lambda_{1}+c)\sqrt{\frac{32\overline{\alpha}^{2}C_{0}(n)}{n\alpha^{2}+(n+p)\beta}}(k+1)^{\frac{1}{n}}
\\&=C_{n,\Omega}(k+1)^{\frac{1}{n}},\end{aligned}
\end{equation*}
where $$C_{n,\Omega}=(\lambda_{1}+c)\sqrt{\frac{32\overline{\alpha}^{2}C_{0}(n)}{n\alpha^{2}+(n+p)\beta}}.$$
Therefore, we complete the proof of theorem \ref{thm1.1}.

$$\eqno\Box$$
\begin{rem}In the theorem \ref{thm1.1}, one can obtain an even stronger result. Indeed, in the proof of this theorem,
there exist a positive integer $1\leq j_{0}\leq n+p$ such that we can choose $n+p$ positive numbers  $\alpha_{1},\alpha_{2},\cdots,\alpha_{n+p}$ satisfy the following:
\begin{equation*}a_{j}^{2}=\|\nabla h_{j}u_{i}\|^{2}=\sqrt{\||\nabla h_{j}|^{2}u_{i}\|^{2}}=b_{j}\geq0,~where~j=1,2,\cdots,j_{0}-1,j_{0}+1,\cdots,n+p, \end{equation*}
\begin{equation*}a_{j_{0}}^{2}=\|\nabla h_{j_{0}}u_{i}\|^{2}\leq\sqrt{\||\nabla h_{j_{0}}|^{2}u_{i}\|^{2}}=b_{j_{0}}\geq0, \end{equation*} and
\begin{equation*}\begin{aligned} \sum_{j=1}^{n+p}\int2 u_{i}\langle\nabla h_{j},\nabla u_{i}\rangle\Delta  h_{j}dv&=0.\end{aligned}\end{equation*}\end{rem}
\begin{corr}
Assume that $(M^{n},g)$ is  an $n$-dimensional complete Riemannian manifolds, which is isometrically
immersed into $(n+\overline{p})$-dimensional Euclidean space $\mathbb{R}^{n+\overline{p}}$. Let $\lambda_{i}$ be the $i$-th $(i=1,2,\cdots,k)$
eigenvalue of the Dirichlet problem \eqref{Eigenvalue-Problem}. Then we have

\begin{equation}\label{z1}\begin{aligned}\lambda_{k+1}-\lambda_{k}\leq
C_{n,\Omega}k^{\frac{1}{n}},\end{aligned}
\end{equation}
where $$C_{n,\Omega}=(\lambda_{1}+c)\sqrt{\frac{32\overline{\alpha}^{2}C_{0}(n)}{n\alpha^{2}+(n+p)\beta}},$$  $C_{0}(n)$ is the same as the one in \eqref{Cheng-Yang-ineq}, and $$
c=\frac{1}{4}\inf_{\psi\in \Psi}\max_{\Omega}n^{2}H^{2}>0,
$$
where $\Psi$ denotes the set of all isometric immersions from $M^n$
into a Euclidean space $\mathbb{R}^{n+\overline{p}}$.
Furthermore,   assume that  $(M^{n},g)$ is an $n$-dimensional complete minimal submanifold which is isometrically
immersed into $(n+p)$-dimensional Euclidean space $\mathbb{R}^{n+p}$, and then the constant $c$ is given by $c=0$.
\end{corr}

\begin{rem}We shall note that when $p$ tends to infinity, it dose not mean that the constant $C_{n,\Omega}$ will be asymptotic to zero. This is because, for any $j=1,2,\cdots,n+p,$ we have

\begin{equation*}a_{j}^{2}=\|\nabla h_{j}u_{i}\|^{2}\geq\sqrt{\||\nabla h_{j}|^{2}u_{i}\|^{2}}=b_{j}\geq0,\end{equation*} which implies that $(n+1)\rho\leq n\alpha^{2}$.  \end{rem}

\begin{rem}Usually, we choose the standard coordinate functions to construct the trial functions to obtain the universal inequalities
or the estimate for the bounds of eigenvalues. However, we do not choose standard coordinate functions but the scarling coordinate functions which satisfy some conditions  to construct the trial functions in the proof of theorem \ref{thm1.1}.\end{rem}
From the proof of theorem \ref{thm1.1}, we have

\begin{rem}If $M^{n}$ is an $n$-dimensional Euclidean space, then we have $H=0$, and thus $c=0$. Let $\alpha_{j}=1$, where $j=1,2,\cdots,n+p$, then $h_{j}=x^{j}$. Thus, we have

$$\alpha=1,$$ and

$$\sum_{j=1}^{n+p}b_{j}=n,$$which implies that

$$C_{n,\Omega}=(\lambda_{1}+c)\sqrt{\frac{32\overline{\alpha}^{2}C_{0}(n)}{n\alpha^{2}+\sum_{j=1}^{n+p}b_{j}}}
=4\lambda_{1}\sqrt{\frac{C_{0}(n)}{n}}.$$
Therefore, the eigenvalue inequality \eqref {z1} in theorem \ref{thm1.1} generalize the eigenvalue inequality \eqref{czy-1} given by Chen-Zheng-Yang in { \rm \cite{CZY}}.
\end{rem}

\textbf{\emph{Proof of theorem}} \ref{thm1.2}. By lemma \ref{lem2.2}, lemma \ref{lem2.4} and lemma \ref{lem3.1}, we can give the proof by using the same method as the proof
of theorem \ref{thm1.1}.
$$\eqno\Box$$
Similarly, we have the following:
\begin{corr}
Let $(M^{n},g)$ be an $n$-dimensional closed Riemannian manifold, which is isometrically immersed into
$(n+\overline{p})$-dimensional Euclidean space $\mathbb{R}^{n+\overline{p}}$, and $\overline{\lambda}_{i}$ be the $i$-th $(i=0,1,2,\cdots,k)$
eigenvalue of the closed eigenvalue problem \eqref{Eigen-Prob-closed}. Then, for any
$k=1,2,\cdots,$ there exist some constants $\alpha^{\prime},$ and $b^{\prime}_{j},j=1,2,\cdots,n+p$, such that

\begin{equation*}\begin{aligned}\overline{\lambda}_{k+1}-\overline{\lambda}_{k}\leq
\overline{C}_{n,\Omega}k^{\frac{1}{n}},\end{aligned}
\end{equation*}
where $$\overline{C}_{n,\Omega}=(\overline{\lambda}_{1}+\overline{c})\sqrt{\frac{32\overline{\alpha}^{\prime2}C_{0}(n)}{n\alpha^{\prime2}+\sum_{j=1}^{n+p}b^{\prime}_{j}}},$$
and $C_{0}(n)$ is the same as the one in \eqref{Cheng-Yang-ineq}, and $$
\overline{c}=\frac{1}{4}\inf_{\psi\in \Psi}\max_{M^{n}}n^{2}H^{2}>0,
$$
where $\Psi$ denotes the set of all isometric immersions from $M^n$
into a Euclidean space $\mathbb{R}^{n+\overline{p}}$.
Furthermore,   assume that  $(M^{n},g)$ is an $n$-dimensional closed minimal submanifold which is isometrically
immersed into $(n+\overline{p})$-dimensional Euclidean space $\mathbb{R}^{n+\overline{p}}$, and then,

\begin{equation}\begin{aligned}\overline{c}=0.\end{aligned}
\end{equation}

\end{corr}

\section{ Estimates for the Eigenvalues on the Unit Sphere and Cylinder}\label{sec4}
In this section, we investigate the eigenvalues on the $n$-dimensional unit sphere $\mathbb{S}^{n}(1)$ and cylinder $\mathbb{R}^{n-m}\times\mathbb{S}^{m}(1)$ with $m<n$.
However, when $n=m$, we assume that $\mathbb{R}^{n-m}\times\mathbb{S}^{m}(1)$ is exactly an $n$-dimensional unit sphere. Under those assumptions, we have
\begin{thm}\label{thm-product-soliton}\ \ Let $M^{n}$ be an $n$-dimensional unit sphere $\mathbb{S}^{n}(1)$ or
cylinder $\mathbb{R}^{n-m}\times\mathbb{S}^{m}(1)$, and $\lambda_{i}$ be the $i$-th $(i=1,2,\cdots,k)$
eigenvalue of the eigenvalue problem \eqref{Eigenvalue-Problem}. Then, we have

\begin{equation}\label{ineq-sp-sl}\begin{aligned}\lambda_{k+1}-\lambda_{k}\leq
C_{n,\Omega}k^{\frac{1}{n}},\end{aligned}
\end{equation}
where $$C_{n,\Omega}=4\left(\lambda_{1}+\frac{m^{2}+mn}{8}\right)\sqrt{\frac{2C_{0}(n)}{n\delta+(n+1)\gamma}}.$$ \end{thm}

\begin{proof}  We denote the position vector of the $n$-dimensional unit round
cylinder $\mathbb{R}^{n-m}\times\mathbb{S}^{m}(1)$ in
$(n+1)$-dimensional Euclidean space $\mathbb{R}^{n+1}$ by $$\textbf{x}
=(\textbf{v},~\textbf{w})= (x^{1},
x^{2},\ldots,x^{n-m},x^{n-m+1},x^{n-m+2} \cdots,x^{n}, x^{n+1}),$$
where $\textbf{v}=(x^{1},
x^{2},\ldots,x^{n-m}),\textbf{w}=(x^{n-m+1},x^{n-m+2} \cdots,x^{n},
x^{n+1})$. In particular, when $n=m$, $\textbf{x}=\textbf{w}$. Then, we obtain
\begin{equation}\label{4.3.54}\sum^{n+1}_{j=n-m+1}(x^{j})^{2}=1,~ \sum^{n+1}_{j=1}|\nabla x^{j}|^{2}=n.\end{equation} It is not difficult to see that, when $n>m$,

\begin{equation}\label{4.3.55}\Delta x^{j}=
\left\{ \begin{aligned}
     &0, &\textnormal{if}\ \ j=1,\cdots,n-m,   \\
                  &-mx^{j},  &\textnormal{if}\ \ j=n-m+1,\cdots,n+1;
                          \end{aligned} \right.
                          \end{equation}and when $n=m$,

\begin{equation}\label{4.3}\Delta x^{j}=-nx^{j}, \textnormal{if}\ \ j=1,\cdots,n+1.\end{equation}
For any $j~(j=1,2,\cdots,n+1)$,   let $l=n+1$ and
$h_{j}(x)=\delta_{j}x^{j}$ and $\delta_{j}>0$, such that

\begin{equation}\begin{aligned}\label{sum-0} \sum_{j=1}^{n+1}\int2 u_{i}\langle\nabla (\delta_{j}x^{j}),\nabla u_{i}\rangle\Delta  (\delta_{j}x^{j})dv&=0,\end{aligned}\end{equation} and

\begin{equation*}a_{j}^{2}=\|\nabla h_{j}u_{i}\|^{2}\geq\sqrt{\||\nabla h_{j}|^{2}u_{i}\|^{2}}=b_{j}\geq0,\end{equation*}
Let $$\delta=\min_{1\leq j\leq n+p}\{\delta_{j}\},$$$$\overline{\delta}=\max_{1\leq j\leq n+p}\{\delta_{j}\},$$
$$\gamma=\min_{1\leq j\leq n+p}\min_{\Omega}\{b_{j}\}.$$
Then, we have

\begin{equation}\begin{aligned}\label{eq-a-b}\sum_{j=1}^{l}\frac{a_{j}^{2}+b_{j}}{2}&=\sum_{j=1}^{n+1}
\frac{\sqrt{\|\nabla (\delta_{j}x^{j})u_{i}\|^{2}}+\sqrt{\||\nabla (\delta_{j}x^{j})|^{2}u_{i}\|^{2}}}{2}\\&\geq\frac{1}{2}\left(n\delta+\sum^{n+1}_{j=1}b_{j}\right)
\\&\geq\frac{1}{2}\left(n\delta+(n+1)\gamma\right).
\end{aligned}\end{equation}
For any fixed point $x_{0}\in\Omega$, we can find a coordinate
system $(\widetilde{x}^{1},\widetilde{x}^{2},\cdots
\widetilde{x}^{n+1})$ of the $n$-dimensional unit round cylinder
$\mathbb{R}^{n-m}\times\mathbb{S}^{m}(1)$ such that at the point
$x_{0}$

\begin{equation}\label{4.3.58}\begin{aligned}&\widetilde{x}^{1}=\cdots=\widetilde{x}^{n}=0,~~\widetilde{x}^{n+1}=1,
\\&\nabla\widetilde{x}^{n+1}=0;~~\nabla_{p}x^{q}=\delta^{q}_{p}~(p,q=1,2,\cdots,n+1).
\end{aligned}\end{equation}
In fact, we can choose a constant $(n+1)\times(n+1)$ type
orthonormal matrix $a^{i}_{j}$
satisfying\begin{equation*}\begin{aligned}\sum^{n+1}_{\alpha=1}a^{\alpha}_{p}a^{\alpha}_{q}=\delta_{pq},
\end{aligned}\end{equation*}
such that
\begin{equation*}\begin{aligned}x^{p}=\sum^{n+1}_{\alpha=1}a^{p}_{\alpha}\widetilde{x}^{\alpha}
\end{aligned}\end{equation*}
and (\ref{4.3.58}) is satisfied at the point $x_{0}$.   By a direct computation, at the
point $x_{0}$,  we yield

\begin{equation*}\begin{aligned}\sum^{n+1}_{p=1}\langle\nabla
x^{p},\nabla
u_{i}\rangle^{2}=|\nabla
u_{i}|^{2}.\end{aligned}\end{equation*} Since $x_{0}$ is an
arbitrary point, we know that for any point $x\in\Omega$,
\begin{equation*}\begin{aligned}\sum^{n+1}_{p=1}\langle\nabla
x^{p},\nabla u_{i}\rangle^{2}=|\nabla
u_{i}|^{2}.\end{aligned}\end{equation*}
On the other hand, by using
(\ref{4.3.54}), we
have\begin{equation}\label{4.3.59}\begin{aligned}&\sum^{n+1}_{p=n-m+1}\nabla(x^{p})^{2}=0,\end{aligned}\end{equation}and
\begin{equation}\label{4.3.60}\begin{aligned}\sum^{n+1}_{p=n-m+1}|\nabla
x^{p}|^{2}=-\sum^{n+1}_{p=1}x^{p}\Delta x^{p}=m.
\end{aligned}\end{equation}
Let

\begin{equation}\begin{aligned}\label{a-ineq}\mathfrak{A}=\sum_{j=1}^{l}\|2\langle\nabla h_{j},\nabla u_{i}\rangle+u_{i}\Delta h_{j}\|^{2}
&=\sum_{j=1}^{n+1}\|2\langle\nabla (\delta_{j}x^{j}),\nabla u_{i}\rangle+u_{i}\Delta(\delta_{j}x^{j})\|^{2}\end{aligned}\end{equation}
Then, by making use of \eqref{4.3.55}, \eqref{sum-0}, \eqref{4.3.59}, \eqref{4.3.60} and \eqref{a-ineq}, we deduce

\begin{equation}\label{U-1}\begin{aligned}\mathfrak{A}&=\sum^{n+1}_{j=1} \|2\langle\nabla
(\delta_{j}x^{j}),\nabla u_{i}\rangle+u_{i}\Delta (\delta_{j}x^{j})\rangle\|_{\Omega}^{2}\\
&=4\sum^{n+1}_{j=1}\int_{\Omega}\langle\nabla
(\delta_{j}x^{j}),\nabla
u_{i}\rangle^{2}dv+m^{2}\sum^{n+1}_{j=n-m+1}\int_{\Omega}u_{i}^{2}(\delta_{j}x^{j})^{2}dv
,\end{aligned}\end{equation}
Furthermore, by \eqref{U-1}, we have

\begin{equation}\label{ineq-a-2}\begin{aligned}\mathfrak{A}
\leq4\overline{\delta}^{2}\sum^{n+1}_{j=1}\int_{\Omega}\langle\nabla
x^{j},\nabla
u_{i}\rangle^{2}d\mu+m^{2}\overline{\delta}^{2}\sum^{n+1}_{j=n-m+1}\int_{\Omega}u_{i}^{2}(x^{j})^{2}d\mu=4\overline{\delta}^{2}\lambda_{i}+m^{2}\overline{\delta}^{2}.\end{aligned}\end{equation}
Let $$c=\rho=\frac{m^{2}}{4}.$$Then, we deduce from \eqref{general-formula-2}, \eqref{eq-a-b} and \eqref{ineq-a-2} that,

\begin{equation}\begin{aligned}\label{general-rmula-3}\frac{1}{2}\left(n\delta+(n+1)\gamma\right)\left(\lambda_{k+2}-\lambda_{k+1}\right)^{2}
\leq4(\lambda_{k+2}+c)\left(4\overline{\delta}^{2}\lambda_{i}+m^{2}\overline{\delta}^{2}\right)
.
\end{aligned}\end{equation}Therefore, by utilizing \eqref{cc2-ineq} and \eqref{general-rmula-3}, we yield
\begin{equation*}\begin{aligned}\lambda_{k+2}-\lambda_{k+1}&\leq
4\overline{\delta}\sqrt{\frac{2}{n\delta+(n+1)\gamma}}\sqrt{\lambda_{1}+c}\sqrt{\lambda_{k+2}+c}\\&\leq
4\overline{\delta}\left(\lambda_{1}+c\right)\sqrt{\frac{2C_{0}(n)}{n\delta+(n+1)\gamma}}(k+1)^{\frac{1}{n}}
\\&=C_{n,\Omega}(k+1)^{\frac{1}{n}},\end{aligned}
\end{equation*}
where $$C_{n,\Omega}=4\overline{\delta}\left(\lambda_{1}+c\right)\sqrt{\frac{2C_{0}(n)}{n\delta+(n+1)\gamma}}.$$ Therefore, we complete the
proof of this theorem.
\end{proof}
\begin{rem} It is easy to see that inequality \eqref{ineq-sp-sl} is also  an intrinsic inequality.
In particular, when $n=m$, inequality \eqref{ineq-sp-sl} gives an intrinsic estimates for the gap of the consecutive eigenvalues on the
sphere space form $\mathbb{S}^{n}(1)$ with unit radius.\end{rem}

\section{Eigenvalues on Complex Projective Spaces}\label{sec5}

In this section, we investigate the eigenvalues of the eigenvalue problem of the Laplacian
on a connected bounded domain  and on a compact complex hypersurface without boundary in the standard
complex projective space $\mathbb{C}P^{n}(4)$ with holomorphic sectional curvature $4$. We shall give an explicit
gap estimate of the consecutive eigenvalues $\lambda_{k+1}-\lambda_{k}$. Firstly, we prove

\begin{thm}\label{thm5.1}Let $\Omega$ a connected bounded domain
in the standard complex projective space $\mathbb{C}P^{n}(4)$ with holomorphic sectional curvature $4$,
and $\lambda_{i}$ be the $i$-th $(i=1,2,\cdots,k)$
eigenvalue of the eigenvalue problem \eqref{Eigenvalue-Problem}. Then, we have

\begin{equation*}\begin{aligned}\lambda_{k+2}-\lambda_{k+1}\leq C(n,\Omega)k^{\frac{1}{2n}},
\end{aligned}\end{equation*} where

\begin{equation*}C(n,\Omega)=4\sqrt{\frac{ 4\left[\left(\sum^{n}_{s=1}\frac{\theta_{s}^{2}}{\theta_{n+1}^{2}} \right)^{2}
+\sum^{n}_{s=1}\frac{\theta_{s}^{4}}{\theta_{n+1}^{4}}\right]+\dfrac{2\overline{\theta}^{2}}{\theta^{2}_{n+1}}\lambda_{1}}
{2\sum^{n}_{s=1}\frac{\theta_{s}^{2}}{\theta_{n+1}^{2}}+\frac{1}{2}(n+1)^{2}\beta}}\cdot\sqrt{C_{0}(n)(\lambda_{1}+2n(n+1))},\end{equation*}
 and $C_{0}(n)$ is the same as the one in \eqref{Cheng-Yang-ineq}. \end{thm}

\begin{proof}
Let $Z=(Z^{1},Z^{2},\cdots,Z^{n+1})$ be a homogeneous coordinate
system of $\mathbb{C}P^{n}(4), (Z^{s}\in \mathbb{C})$. Defining $\Psi_{s\overline{t}}$, for $s,t = 1, 2,\cdots,n+1,$ by

\begin{equation*}\Psi_{s\overline{t}}(\theta_{1},\cdots,\theta_{n+1},Z^{1},\cdots,Z^{n+1})
=\frac{(\theta_{s}Z^{s})\overline{(\theta_{t}Z^{t})}}{\sum^{n+1}_{r=1}(\theta_{r}Z^{r})\overline{(\theta_{r}Z^{r}})},
\end{equation*}where $\theta_{s}, s= 1,2,\cdots,n+1,$ are $(n+1)$ coefficients of scarling coordinate system which are determined later,
we have

\begin{equation*}\Psi_{s\overline{t}} = \overline{\Psi_{t\overline{s}}},\ \ \ \sum^{n+1}_{s,t=1}\Psi_{s\overline{t}}\overline{\Psi_{s\overline{t}}} = 1.\end{equation*}
For any fixed point $P\in M^{n}$, we can choose a new homogeneous coordinate system of
$\mathbb{C}P^{n}(4)$, which satisfies that, at $P$

\begin{equation}\label{Z}\widetilde{Z}^{1}=\widetilde{Z}^{2}=\cdots=\widetilde{Z}^{n}=0,\ \ \ \ \widetilde{Z}^{n+1}\neq0
\end{equation} and

\begin{equation*}Z^{s}=\sum^{n+1}_{r=1}A_{sr}\widetilde{Z}^{r},
\end{equation*}
where $$A= (A_{st})\in U(n + 1)$$ is an $(n+1)\times(n+1)$-unitary matrix, that is, $A_{st}$ satisfies

\begin{equation*} \sum^{n+1}_{r=1}A_{rs}\overline{A_{rt}}=\sum^{n+1}_{r=1}A_{sr}\overline{A_{tr}}=\delta_{st}.
\end{equation*}
Then, we know that $$z = (z^{1},\cdots,z^{n}),\ \ z^{s}=\frac{\theta_{s}\widetilde{Z}^{s}}{\theta_{n+1}\widetilde{Z}^{n+1}},$$ is a local holomorphic coordinate
system of $\mathbb{C}P^{n}(4)$ in a neighborhood $U$ of the point $P\in M^{n}$ and \eqref{Z} implies that,  at $P$,

\begin{equation*}z^{1}=\cdots=z^{n}=0.\end{equation*}
Hence, we infer, for $s,t = 1,2,\cdots,n+1,$

\begin{equation}\begin{aligned}\label{f-funs-1-1}&\widetilde{\Psi}_{s\overline{t}}(\theta_{1},\cdots,\theta_{n+1},Z^{1},\cdots,Z^{n+1})
=\frac{(\theta_{s}\widetilde{Z}^{s})\overline{(\theta_{t}\widetilde{Z}^{t})}}
{\sum^{n+1}_{r=1}(\theta_{r}\widetilde{Z}^{r})\overline{(\theta_{r}\widetilde{Z}^{r}})}
=\frac{z^{s}\overline{z^{t}}}{1+\sum^{n}_{r=1}z^{r}\overline{z^{r}}}\\
&\Psi_{s\overline{t}}=\sum^{n+1}_{r,v=1}A_{sr}\overline{A_{tv}}\widetilde{\Psi}_{r\overline{v}},\ \ z^{n+1}\equiv1.\end{aligned}\end{equation}
Putting $$\mathcal{G}_{s\overline{t}}(\theta_{1},\cdots,\theta_{n+1},Z^{1},\cdots,Z^{n+1})=
Re(\Psi_{s\overline{t}}(\theta_{1},\cdots,\theta_{n+1},Z^{1},\cdots,Z^{n+1}))$$ and

$$\mathcal{F}_{s\overline{t}}(\theta_{1},\cdots,\theta_{n+1},Z^{1},\cdots,Z^{n+1})={\rm Im}(\Psi_{s\overline{t}}(\theta_{1},\cdots,\theta_{n+1},Z^{1},\cdots,Z^{n+1})),$$
for $s,t=1,2,\cdots,n+1$, then, we have

\begin{equation}\begin{aligned}\label{G-F=1-1}&\sum^{n+1}_{s,t=1}\left(\mathcal{G}^{2}_{s\overline{t}}+\mathcal{F}^{2}_{s\overline{t}}\right)
=\sum^{n+1}_{s,t=1}\Psi_{s\overline{t}}\overline{\Psi_{s\overline{t}}}
=\sum^{n+1}_{s,t=1}\widetilde{\Psi}_{p\overline{q}}\overline{\widetilde{\Psi}_{p\overline{q}}}=1\\
&\sum^{n+1}_{s,t=1}\left(\mathcal{G}_{s\overline{t}}\nabla \mathcal{G}_{s\overline{t}}+\mathcal{F}_{s\overline{t}}\nabla \mathcal{F}_{s\overline{t}}\right)=0.\end{aligned}\end{equation}
Next, according to special proportion, we define the corresponding weighted transformation (or we say that they are some scarling transformations) to the variables $z^{s}$, where $s=1,2,\cdots,n+1$, as follows: $$y^{s}=\frac{\theta_{n+1}}{\theta_{s}}z^{s}.$$
We note that those weighted transformations play a significant role in the calculation. Then, it follows from \eqref{f-funs-1-1} that

\begin{equation}\begin{aligned}\label{f-funs-12}&\widetilde{\Psi}_{s\overline{t}}(\theta_{1},\cdots,\theta_{n+1},Z^{1},\cdots,Z^{n+1})
=\frac{\theta_{s}\theta_{t}y^{s}\overline{y^{t}}}{\theta^{2}_{n+1}+\sum^{n}_{r=1}\theta^{2}_{r}y^{r}\overline{y^{r}}},\end{aligned}\end{equation}and
\begin{equation}\begin{aligned}\label{f-funs-2-n+1}\widetilde{\Psi}_{(n+1)\overline{(n+1)}}(\theta_{1},\cdots,\theta_{n+1},Z^{1},\cdots,Z^{n+1})
=1\end{aligned}\end{equation}
Let
$$g=\sum^{n}_{s,t=1} g_{s\overline{t}}dy^{s}d\overline{y^{t}}$$
be the Fubini-Study metric of $\mathbb{C}P^{n}(4)$. Then,

\begin{equation}\begin{aligned}\label{g-st}&g_{s\overline{t}}=\frac{\delta_{s\overline{t}}}{1+\sum^{n}_{r=1}
|y^{r}|^{2}}-\frac{y^{t}\overline{y^{s}}}{\left(1+\sum^{n}_{r=1}|y^{r}|^{2}\right)^{2}}\\
&\left(g_{s\overline{t}}\right)^{-1}=\left(g^{s\overline{t}}\right)\\
&g^{s\overline{t}}=\left(1+\sum^{n}_{r=1}|y^{r}|^{2}\right)(\delta_{s\overline{t}}+y^{t}\overline{y^{s}}).\end{aligned}\end{equation}
Under the local coordinate system, for any smooth function $\Psi$, it follows from \eqref{g-st} that

\begin{equation*}\begin{aligned}\Delta \Psi=4\sum^{n}_{s,t=1}g^{s\overline{t}}\frac{\partial^{2}}{\partial y^{s}\overline{\partial y^{t}}} \Psi,\end{aligned}\end{equation*}
And, by the definition of $\widetilde{\Psi}_{s\overline{t}}$, we know that, at $P$,

\begin{equation*}\begin{aligned}&\Delta =4\sum^{n}_{r=1}\frac{\partial^{2}}{\partial y^{r}\overline{\partial y^{r}}},\\
&\nabla \widetilde{\Psi}_{s\overline{t}}= 0, \ \ {\rm  if}\ \  s\neq0\ \  {\rm  and} \ \ t\neq 0\\
&\nabla \widetilde{\Psi}_{s\overline{s}}= 0,\\
&\Delta \widetilde{\Psi}_{s\overline{t}}= 0, \ \ {\rm if} \ \ s\neq t\\
&\Delta \widetilde{\Psi}_{(n+1)\overline{(n+1)}}=-4\sum^{n}_{s=1}\frac{\theta_{s}^{2}}{\theta_{n+1}^{2}}; \Delta \widetilde{\Psi}_{r\overline{r}}=\frac{4\theta_{r}^{2}}{\theta_{n+1}^{2}},\ \ r=1,\cdots,n.
\end{aligned}\end{equation*}
Thus, we obtain from \eqref{G-F=1-1}, \eqref{f-funs-12} and  \eqref{f-funs-2-n+1}, at $P$,

\begin{equation}\begin{aligned}\label{sys-1}\sum^{n+1}_{s,t=1}\left(\nabla \mathcal{G}_{s\overline{t}}\cdot
\nabla \mathcal{G}_{s\overline{t}}+\nabla \mathcal{F}_{s\overline{t}}\cdot\nabla \mathcal{F}_{s\overline{t}}\right)&=
-\sum^{n+1}_{s,t=1}\left(\mathcal{G}_{s\overline{t}}\Delta \mathcal{G}_{s\overline{t}}+\mathcal{F}_{s\overline{t}}\Delta \mathcal{F}_{s\overline{t}}\right)\\
&=-{\rm Re}\sum^{n+1}_{s,t=1}\overline{\Psi_{s\overline{t}}}\Delta \Psi_{s\overline{t}}\\&
=-{\rm Re}\sum^{n+1}_{s,t=1}\sum^{n+1}_{r,w=1}\overline{A_{sr}\overline{A_{tw}}}\overline{\widetilde{\Psi}_{r\overline{w}}}
\sum^{n+1}_{u,v=1}A_{su}\overline{A_{tv}}\Delta\widetilde{\Psi}_{u\overline{v}}\\
&=-\sum^{n+1}_{s,t=1}{\rm Re}\overline{\widetilde{\Psi}_{s\overline{t}}}\Delta \widetilde{\Psi}_{s\overline{t}}\\&
=-\widetilde{\Psi}_{(n+1)\overline{(n+1)}}\Delta\widetilde{\Psi}_{(n+1)\overline{(n+1)}}\\&=4\sum^{n}_{s=1}\frac{\theta_{s}^{2}}{\theta_{n+1}^{2}}.
\end{aligned}\end{equation}
By a similar calculation, we have, at $P$,

\begin{equation}\begin{aligned}\label{sys-2}\sum^{n+1}_{s,t=1}\left(\nabla \mathcal{G}_{s\overline{t}}
\Delta \mathcal{G}_{s\overline{t}}+\nabla \mathcal{F}_{s\overline{t}}\Delta \mathcal{F}_{s\overline{t}}\right)
={\rm Re}\sum^{n+1}_{s,t=1}\nabla\overline{\Psi_{s\overline{t}}}\Delta \Psi_{s\overline{t}}=0.
\end{aligned}\end{equation}

\begin{equation}\begin{aligned}\label{sys-3}\sum^{n+1}_{s,t=1}\left(\Delta \mathcal{G}_{s\overline{t}}
\Delta \mathcal{G}_{s\overline{t}}+\Delta \mathcal{F}_{s\overline{t}}\Delta \mathcal{F}_{s\overline{t}}\right)
&={\rm Re}\sum^{n+1}_{s,t=1}\overline{\Delta \Psi_{s\overline{t}}}\Delta \Psi_{s\overline{t}}\\
&={\rm Re}\sum^{n+1}_{s,t=1}\overline{\Delta \widetilde{\Psi}_{s\overline{t}}}\Delta \widetilde{\Psi}_{s\overline{t}}
\\&=\left(-4\sum^{n}_{s=1}\frac{\theta_{s}^{2}}{\theta_{n+1}^{2}}\right)\times\left(-4\sum^{n}_{s=1}\frac{\theta_{s}^{2}}{\theta_{n+1}^{2}}\right)
+4\times 4\times\sum^{n}_{s=1}\frac{\theta_{s}^{4}}{\theta_{n+1}^{4}}
\\&=16\left(\sum^{n}_{s=1}\frac{\theta_{s}^{2}}{\theta_{n+1}^{2}} \right)^{2}+16\sum^{n}_{s=1}\frac{\theta_{s}^{4}}{\theta_{n+1}^{4}}.
\end{aligned}\end{equation}and

\begin{equation}\begin{aligned}\label{sys-4}\sum^{n+1}_{s,t=1}\left(\langle\nabla \mathcal{G}_{s\overline{t}},\nabla u_{i}\rangle^{2}
+\langle\nabla \mathcal{F}_{s\overline{t}},\nabla u_{i}\rangle^{2}\right)
&={\rm Re}\sum^{n+1}_{s,t=1}\overline{\langle\nabla\overline{\Psi_{s\overline{t}}},\nabla u_{i}\rangle}\langle\nabla \Psi_{s\overline{t}},\nabla u_{i}\rangle\\
&={\rm Re}\sum^{n+1}_{s,t=1}\overline{\langle\nabla\overline{\widetilde{\Psi}_{s\overline{t}}},
\nabla u_{i}\rangle}\langle\nabla \widetilde{\Psi}_{s\overline{t}},\nabla u_{i}\rangle\\
&\leq\frac{2\overline{\theta}^{2}}{\theta^{2}_{n+1}}|\nabla u_{i}|^{2},
\end{aligned}\end{equation}where$$\overline{\theta}
=\max_{1\leq s\leq n+1}\{\theta_{s}\}.$$Since $P$ is arbitrary, we have at any point $x \in M^{n}$,

\begin{equation}\left\{\begin{aligned}\label{sys-5}&\sum^{n+1}_{s,t=1}\left(\nabla \mathcal{G}_{s\overline{t}}\cdot\nabla \mathcal{G}_{s\overline{t}}
+\nabla \mathcal{F}_{s\overline{t}}\cdot\nabla \mathcal{F}_{s\overline{t}}\right)=4\sum^{n}_{s=1}\frac{\theta_{s}^{2}}{\theta_{n+1}^{2}}.
\\&\sum^{n+1}_{s,t=1}\left(\nabla \mathcal{G}_{s\overline{t}}\Delta \mathcal{G}_{s\overline{t}}+\nabla \mathcal{F}_{s\overline{t}}\Delta \mathcal{F}_{s\overline{t}}\right)=0.
\\&\sum^{n+1}_{s,t=1}\left(\Delta \mathcal{G}_{s\overline{t}}\Delta \mathcal{G}_{s\overline{t}}+\Delta \mathcal{F}_{s\overline{t}}\Delta \mathcal{F}_{s\overline{t}}\right)=
16\left[\left(\sum^{n}_{s=1}\frac{\theta_{s}^{2}}{\theta_{n+1}^{2}} \right)^{2}+\sum^{n}_{s=1}\frac{\theta_{s}^{4}}{\theta_{n+1}^{4}}\right].
\\&\sum^{n+1}_{s,t=1}\left(\langle\nabla \mathcal{G}_{s\overline{t}},\nabla u_{i}\rangle^{2}
+\langle\nabla \mathcal{F}_{s\overline{t}},\nabla u_{i}\rangle^{2}\right)\leq\frac{2\overline{\theta}^{2}}{\theta^{2}_{n+1}}|\nabla u_{i}|^{2}.
\end{aligned}\right.\end{equation}
By applying the Lemma \ref{lem2.3} to the functions $\mathcal{G}_{s\overline{t}}$ and $\mathcal{F}_{s\overline{t}}$ and taking sum on $s$ and $t$
from $1$ to $n+1$, we infer from \eqref{sys-1}, \eqref{sys-2}, \eqref{sys-3}, \eqref{sys-4} and \eqref{sys-5} that

\begin{equation}\begin{aligned}\label{sum-pq-1}
\sum^{n+1}_{s,t=1}\left(\|u_{i}\nabla \mathcal{G}_{s\overline{t}}\|^{2}+\|u_{i}\nabla \mathcal{F}_{s\overline{t}}\|^{2}\right)&=
4\int_{M^{n}}\sum^{n+1}_{s,t=1}\left(\langle\nabla \mathcal{G}_{s\overline{t}}\Delta G_{s\overline{t}},u_{i}\nabla u_{i}\rangle
+\langle\nabla \mathcal{F}_{s\overline{t}}\Delta \mathcal{F}_{s\overline{t}},u_{i}\nabla u_{i}\rangle\right)dv
\\&+\int_{M^{n}}\sum^{n+1}_{s,t=1}\left(\Delta \mathcal{G}_{s\overline{t}}\Delta \mathcal{G}_{s\overline{t}}
+\Delta \mathcal{F}_{s\overline{t}}\Delta \mathcal{F}_{s\overline{t}}\right)u_{i}^{2}dv
\\&+4\int_{M^{n}}\sum^{n+1}_{s,t=1}\left(\langle\nabla \mathcal{G}_{s\overline{t}},\nabla u_{i}\rangle^{2}
+\langle\nabla \mathcal{F}_{s\overline{t}},\nabla u_{i}\rangle^{2}\right)dv\\&
\leq16\left[\left(\sum^{n}_{s=1}\frac{\theta_{s}^{2}}{\theta_{n+1}^{2}} \right)^{2}+\sum^{n}_{s=1}\frac{\theta_{s}^{4}}{\theta_{n+1}^{4}}\right]+\frac{8\overline{\theta}^{2}}{\theta^{2}_{n+1}}\lambda_{i}.
\end{aligned}\end{equation}We choose $(n+1)$ positive real numbers $\theta_{s}$, such that, for all $s=1,2,\cdots,n+1$,
\begin{equation}\begin{aligned}\label{5.20} \sum_{s,t=1}^{n+1}\left[\int2 u_{i}\langle\nabla \mathcal{G}_{s\overline{t}},\nabla u_{i}\rangle\Delta  \mathcal{G}_{s\overline{t}}
dv+\int2 u_{i}\langle\nabla \mathcal{F}_{s\overline{t}},\nabla u_{i}\rangle\Delta  \mathcal{F}_{s\overline{t}}
dv\right]&=0,\end{aligned}\end{equation}and
\begin{equation*}a_{s\overline{t}}^{2}=\|\nabla \mathcal{G}_{s\overline{t}}u_{i}\|^{2}+\|\nabla \mathcal{F}_{s\overline{t}}u_{i}\|^{2}
\geq\sqrt{\||\nabla \mathcal{G}_{s\overline{t}}|^{2}u_{i}\|^{2}}+\sqrt{\||\nabla \mathcal{F}_{s\overline{t}}|^{2}u_{i}\|^{2}}=b_{s\overline{t}}\geq0.\end{equation*}
Let  $$\beta=\min_{1\leq s,t\leq n+1}\{b_{s\overline{t}}\},$$
and $l=n+1$, then, by lemma \ref{lem2.1}, we have

\begin{equation}\begin{aligned}\label{a-b-1}\sum^{n+1}_{s,t=1}\frac{a_{s\overline{t}}^{2}+b_{s\overline{t}}}{2}&=2\sum^{n}_{s=1}\frac{\theta_{s}^{2}}{\theta_{n+1}^{2}}
+\sum_{s,t=1}^{n+1}\frac{b_{s\overline{t}}}{2}\\&\geq2\sum^{n}_{s=1}\frac{\theta_{s}^{2}}{\theta_{n+1}^{2}}+\frac{1}{2}(n+1)^{2}\beta.\end{aligned}\end{equation}
From \eqref{general-formula-2}, \eqref{sum-pq-1}, \eqref{5.20} and \eqref{a-b-1}, we obtain
\begin{equation*}\begin{aligned}\left[2\sum^{n}_{s=1}\frac{\theta_{s}^{2}}{\theta_{n+1}^{2}}+\frac{1}{2}(n+1)^{2}\beta\right]\left(\lambda_{k+2}-\lambda_{k+1}\right)^{2}
\leq4(\lambda_{k+2}+\rho)\left[16\sum^{n}_{s=1}\frac{\theta_{s}^{2}}{\theta_{n+1}^{2}}\left(\sum^{n}_{s=1}\frac{\theta_{s}^{2}}{\theta_{n+1}^{2}} + 1\right)
+\frac{8\overline{\theta}^{2}}{\theta^{2}_{n+1}}\lambda_{i}\right],
\end{aligned}\end{equation*}for any $i,~i=1,2,\cdots,k$, which implies
 \begin{equation}\begin{aligned}\label{po-gap-1}\lambda_{k+2}-\lambda_{k+1}
\leq\sqrt{\frac{4(\lambda_{k+2}+\rho)\left\{16\left[\left(\sum^{n}_{s=1}\frac{\theta_{s}^{2}}{\theta_{n+1}^{2}} \right)^{2}
+\sum^{n}_{s=1}\frac{\theta_{s}^{4}}{\theta_{n+1}^{4}}\right]+\dfrac{8\overline{\theta}^{2}}{\theta^{2}_{n+1}}\lambda_{1}\right\}}
{2\sum^{n}_{s=1}\frac{\theta_{s}^{2}}{\theta_{n+1}^{2}}+\frac{1}{2}(n+1)^{2}\beta}}.
\end{aligned}\end{equation} In order to complete the proof, we need a recursion formula given by Cheng and Yang in \cite{CY3} as follows:
Let  $\mu_1 \leq  \mu_2 \leq  \dots, \leq \mu_{k+1}$ be any positive
real numbers satisfying
\begin{equation*}
  \sum_{i=1}^k(\mu_{k+1}-\mu_i)^2 \le
 \frac 4n\sum_{i=1}^k\mu_i(\mu_{k+1} -\mu_i).
\end{equation*}
 Define
 \begin{equation*}
 \Gamma_k=\frac 1k\sum_{i=1}^k\mu_i,\qquad \mathcal{E}_k=\frac 1k
\sum_{i=1}^k\mu_i^2, \ \ \ \mathcal{H}_k=\left (1+\frac 2n \right
)\Gamma_k^2-\mathcal{E}_k.
\end{equation*}
Then, we have
\begin{equation}\label{recursion}
\mathcal{H}_{k+1}\leq C(n,k) \left ( \frac {k+1}k \right )^{\frac 4n}\mathcal{H}_k,
\end{equation}
where
$$
C(n,k) =1-\frac 1{3n}
  \left (\frac k{k+1}\right )^{\frac
  4n}\frac {\left(1+\frac 2n\right )\left (1+
  \frac 4n\right )}{(k+1)^3}<1.
$$
 Let  $\mu_{i}=\lambda_{i}+2n(n+1)$. By \eqref{cy-universal-1} and \eqref{recursion}, we yield

\begin{equation*}\begin{aligned}
\sum^{k}_{i=1}\left(\mu_{k+1}-\mu_{i}\right)^{2}\leq\frac{4}{2n}\sum^{k}_{i=1}\left(\mu_{k+1}-\mu_{i}\right)\mu_{i},
\end{aligned}\end{equation*}
which implies

\begin{equation}\begin{aligned}\label{cy-upper-1-1}
\lambda_{k+1}+2n(n+1)\leq C_{0}(n)(\lambda_{1}+2n(n+1))k^{\frac{1}{n}},
\end{aligned}\end{equation} $C_{0}(n)$ is the same as the one in \eqref{Cheng-Yang-ineq}.
Therefore, putting $\rho=2n(n+1)$  and synthesizing \eqref{po-gap-1} and \eqref{cy-upper-1-1}, we obtain

\begin{equation*}\begin{aligned}\lambda_{k+2}-\lambda_{k+1}
&\leq\sqrt{\frac{4(\lambda_{k+2}+2n(n+1))\left\{16\left[\left(\sum^{n}_{s=1}\frac{\theta_{s}^{2}}{\theta_{n+1}^{2}} \right)^{2}
+\sum^{n}_{s=1}\frac{\theta_{s}^{4}}{\theta_{n+1}^{4}}\right]
+\dfrac{8\overline{\theta}^{2}}{\theta^{2}_{n+1}}\lambda_{1}\right\}}
{2\sum^{n}_{s=1}\frac{\theta_{s}^{2}}{\theta_{n+1}^{2}}+\frac{1}{2}(n+1)^{2}\beta}}\\
&\leq4\sqrt{\frac{ 4\left[\left(\sum^{n}_{s=1}\frac{\theta_{s}^{2}}{\theta_{n+1}^{2}} \right)^{2}
+\sum^{n}_{s=1}\frac{\theta_{s}^{4}}{\theta_{n+1}^{4}}\right]+\dfrac{2\overline{\theta}^{2}}{\theta^{2}_{n+1}}\lambda_{1}}
{2\sum^{n}_{s=1}\frac{\theta_{s}^{2}}{\theta_{n+1}^{2}}+\frac{1}{2}(n+1)^{2}\beta}}\cdot\sqrt{C_{0}(n)(\lambda_{1}+2n(n+1))}\cdot (k+1)^{\frac{1}{2n}}\\
&=C(n,\Omega)k^{\frac{1}{2n}},
\end{aligned}\end{equation*} where

\begin{equation*}C(n,\Omega)=4\sqrt{\frac{ 4\left[\left(\sum^{n}_{s=1}\frac{\theta_{s}^{2}}{\theta_{n+1}^{2}} \right)^{2}
+\sum^{n}_{s=1}\frac{\theta_{s}^{4}}{\theta_{n+1}^{4}}\right]+\dfrac{2\overline{\theta}^{2}}{\theta^{2}_{n+1}}\lambda_{1}}
{2\sum^{n}_{s=1}\frac{\theta_{s}^{2}}{\theta_{n+1}^{2}}+\frac{1}{2}(n+1)^{2}\beta}}\cdot\sqrt{C_{0}(n)(\lambda_{1}+2n(n+1))}.\end{equation*}
Therefore, we finish the proof of this theorem.

\end{proof}

Next, we shall consider the eigenvalue problem of the Laplacian on a
compact complex hypersurface $M^{n}$ without boundary in $\mathbb{C}P^{n+1}(4)$:

\begin{thm}\label{thm5.2} Let $M^{n}$ a compact complex hypersurface with empty boundary
in the standard complex projective space $\mathbb{C}P^{n}(4)$ with holomorphic sectional curvature $4$,
and $\overline{\lambda}_{i}$ be the $i$-th $(i=1,2,\cdots,k)$
eigenvalue of the eigenvalue problem \eqref{Eigen-Prob-closed}. Then, we have

\begin{equation}\begin{aligned}\overline{\lambda}_{k+2}-\overline{\lambda}_{k+1}\leq C(n,M^{n})(k+1)^{\frac{1}{2n}},
\end{aligned}\end{equation} where

\begin{equation*}C(n,M^{n})=\sqrt{\frac{4\left[\left(\sum^{n}_{s=1}\frac{\theta_{s}^{2}}{\theta_{n+2}^{2}}\right)^{2}+\sum^{n}_{s=1}\frac{\theta_{s}^{4}}{\theta_{n+2}^{4}}\right]
+\dfrac{2\overline{\theta}^{2}}{\theta^{2}_{n+2}}\overline{\lambda}_{1}}
{2\sum^{n}_{s=1}\frac{\theta_{s}^{2}}{\theta_{n+1}^{2}}+\frac{1}{2}(n+1)^{2}\beta}}\cdot\sqrt{C_{0}(n)(\overline{\lambda}_{1}+2n(n+1))},\end{equation*}
 and $C_{0}(n)$ is the same as the one in \eqref{Cheng-Yang-ineq}. \end{thm}

\begin{proof}Since the method of proof is the same as in the proof of
Theorem 1, we shall only give its outline.
Let $Z=(Z^{1},Z^{2},\cdots,Z^{n+2})$ be a homogeneous coordinate
system of $\mathbb{C}P^{n+1}(4), (Z_{s}\in \mathbb{C})$. Defining $\Psi_{s\overline{t}}$, for $s,t = 1, 2,\cdots,n+2,$ by

\begin{equation*}\Psi_{s\overline{t}}(\theta_{1},\cdots,\theta_{n+2},Z^{1},\cdots,Z^{n+2})=\frac{(\theta_{s}Z^{s})
\overline{(\theta_{t}Z^{t})}}{\sum^{n+2}_{r=1}(\theta_{r}Z^{r})\overline{(\theta_{r}Z^{r}})},
\end{equation*}where $\theta_{s}, s= 1,2,\cdots,n+2,$ are $(n+2)$ coefficients of scarling coordinate system, which will be determined later,
we have

\begin{equation*}\Psi_{s\overline{t}} = \overline{\Psi_{t\overline{s}}},\ \ \ \sum^{n+2}_{s,t=1}\Psi_{s\overline{t}}\overline{\Psi_{s\overline{t}}} = 1.\end{equation*}
For any fixed point $P\in M^{n}$, we can choose a new homogeneous coordinate system of
$\mathbb{C}P^{n}(4)$, which satisfies, at $P$

\begin{equation*}\widetilde{Z}^{1}=\widetilde{Z}^{2}=\cdots=\widetilde{Z}^{n+1}=0,\ \ \ \ \widetilde{Z}^{n+2}\neq0
\end{equation*} and

\begin{equation*}Z^{s}=\sum^{n+2}_{r=1}C_{sr}\widetilde{Z}^{r},
\end{equation*}
where $A= (A_{st})\in U(n + 2)$ is an $(n+2)\times(n+2)$-unitary matrix, that is, $A_{st}$ satisfies

\begin{equation*} \sum^{n+2}_{r=1}A_{rs}\overline{A_{rt}}=\sum^{n+2}_{r=1}A_{sr}\overline{A_{tr}}=\delta_{st}.
\end{equation*}Let
$$z^{s}=\frac{\theta_{s}\widetilde{Z}^{s}}{\theta_{n+2}\widetilde{Z}^{n+2}},\ \ {\rm for}\ \ s=1,2,\cdots,n+2,$$
Then, we know that $z = (z^{1},\cdots,z^{n})$ is a local holomorphic coordinate
system of $M$ in a neighborhood $U$ of the point $P\in M$ and $z^{n+1}=h(z^{1},z^{2},\cdots,z^{n})$
is a holomorphic function of $z^{1},z^{2},\cdots,z^{n}$ and satisfying,

\begin{equation*}\frac{\partial h}{\partial z^{s}} \Bigg{|}_{P}=0,\ \ {\rm for}\ \ s=1,2,\cdots,n,
\end{equation*}
At the point $P$, one has

\begin{equation*}z^{1}=\cdots=z^{n+1}=0.\end{equation*}
Hence, for any $s,t = 1,2,\cdots,n+2,$ we have,

\begin{equation}\begin{aligned}\label{f-funs-1}&\widetilde{\Psi}_{s\overline{t}}(\theta_{1},\cdots,
\theta_{n+2},Z^{1},\cdots,Z^{n+2})=\frac{(\theta_{s}\widetilde{Z}^{s})\overline{(\theta_{t}\widetilde{Z}^{t})}}
{\sum^{n+2}_{r=1}(\theta_{r}\widetilde{Z}^{r})\overline{(\theta_{r}\widetilde{Z}^{r}})}
=\frac{z^{s}\overline{z^{t}}}{1+\sum^{n+1}_{r=1}z^{r}\overline{z^{r}}}\\
&\Psi_{s\overline{t}}=\sum^{n+2}_{r,v=1}A_{sr}\overline{A_{tv}}\widetilde{\Psi}_{r\overline{v}},\ \ z^{n+2}\equiv1.\end{aligned}\end{equation}
Putting $$\mathcal{G}_{s\overline{t}}(\theta_{1},\cdots,\theta_{n+2},Z^{1},\cdots,Z^{n+2})= Re(\Psi_{s\overline{t}}(\theta_{1},\cdots,\theta_{n+2},Z^{1},\cdots,Z^{n+2}))$$
and $$\mathcal{F}_{s\overline{t}}(\theta_{1},\cdots,\theta_{n+2},Z^{1},\cdots,Z^{n+2})
={\rm Im}(\Psi_{s\overline{t}}(\theta_{1},\cdots,\theta_{n+2},Z^{1},\cdots,Z^{n+2})),$$ for $s,t=1,2,\cdots,n+2$, then, we infer

\begin{equation*}\begin{aligned}&\sum^{n+2}_{s,t=1}\left(\mathcal{G}^{2}_{s\overline{t}}+\mathcal{F}^{2}_{s\overline{t}}\right)
=\sum^{n+2}_{s,t=1}\Psi_{s\overline{t}}\overline{\Psi_{s\overline{t}}}
=\sum^{n+2}_{s,t=1}\widetilde{\Psi}_{p\overline{q}}\overline{\widetilde{\Psi}_{p\overline{q}}}=1\\
&\sum^{n+2}_{s,t=1}\left(\mathcal{G}_{s\overline{t}}\nabla \mathcal{G}_{s\overline{t}}+\mathcal{F}_{s\overline{t}}
\nabla \mathcal{F}_{s\overline{t}}\right)=0.\end{aligned}\end{equation*}
Similarly, we define the corresponding weighted transformations of the variables $z^{s}$, where $s=1,2,\cdots, n+2,$ as follows:  $$y^{s}=\frac{\theta_{n+2}}{\theta_{s}}z^{s}.$$ Then, it follows from \eqref{f-funs-1} that

\begin{equation}\begin{aligned}\label{f-funs-2}&\widetilde{\Psi}_{s\overline{t}}
=\frac{\theta_{s}\theta_{t}y^{s}\overline{y^{t}}}{\theta^{2}_{n+2}+\sum^{n+1}_{r=1}\theta^{2}_{r}y^{r}\overline{y^{r}}}.\end{aligned}\end{equation}
It is easy to see that, under the local coordinate system, for $z\in U$, the metric can be written as the following:

$$g_{M}=\sum^{n}_{s,t=1}\left(1+O\left(|z|^{2}\right)\right)dz^{s}d\overline{z^{t}},$$Thus, for any smooth function $\Psi$, we have

\begin{equation*}\begin{aligned}\Delta \Psi=4\sum^{n}_{s,t=1}g^{s\overline{t}}\frac{\partial^{2}}{\partial y^{s}\overline{\partial y^{t}}} \Psi,\end{aligned}\end{equation*}
By a direct calculation, we obtain, at $P$,

\begin{equation*}\begin{aligned}&\Delta =4\sum^{n}_{r=1}\frac{\partial^{2}}{\partial y^{r}\overline{\partial y^{r}}},\\
&\nabla \widetilde{\Psi}_{s\overline{t}}= 0, \ \ {\rm  if}\ \  s\neq n+2\ \  {\rm  and} \ \ t\neq n+2\\
&\nabla \widetilde{\Psi}_{s\overline{s}}= 0,\\
&\nabla \widetilde{\Psi}_{\overline{s}(n+2)}= \nabla \widetilde{\Psi}_{(n+2)\overline{s}}
= \nabla \widetilde{\Psi}_{(n+2)\overline{(n+2)}}= 0, \ \ {\rm for}\ \ s=1,2,\cdots,n,\\
&\Delta \widetilde{\Psi}_{s\overline{t}}= 0, \ \ {\rm if} \ \ s\neq t,\ \ \Delta \widetilde{\Psi}_{(n+1)\overline{(n+1)}} =0,\\
&\Delta \widetilde{\Psi}_{(n+2)\overline{(n+2)}}=-4\sum^{n}_{s=1}\frac{\theta_{s}^{2}}{\theta_{n+2}^{2}};
\ \  \Delta \widetilde{\Psi}_{r\overline{r}}=\frac{4\theta_{r}^{2}}{\theta_{n+2}^{2}},\ \ r=1,\cdots,n.
\end{aligned}\end{equation*}
Similarly, one can check the following:
\begin{equation}\left\{\begin{aligned}\label{z-a}&\sum^{n+2}_{s,t=1}\left(\nabla \mathcal{G}_{s\overline{t}}\cdot\nabla \mathcal{G}_{s\overline{t}}
+\nabla \mathcal{F}_{s\overline{t}}\cdot\nabla \mathcal{F}_{s\overline{t}}\right)=4\sum^{n}_{s=1}\frac{\theta_{p}^{2}}{\theta_{n+2}^{2}};
\\&\sum^{n+2}_{s,t=1}\left(\nabla \mathcal{G}_{s\overline{t}}\Delta \mathcal{G}_{s\overline{t}}+\nabla \mathcal{F}_{s\overline{t}}\Delta \mathcal{F}_{s\overline{t}}\right)=0;
\\&\sum^{n+2}_{s,t=1}\left(\Delta \mathcal{G}_{s\overline{t}}\Delta \mathcal{G}_{s\overline{t}}+\Delta \mathcal{F}_{s\overline{t}}\Delta \mathcal{F}_{s\overline{t}}\right)=
16\left[\left(\sum^{n}_{s=1}\frac{\theta_{s}^{2}}{\theta_{n+2}^{2}}\right)^{2}+\sum^{n}_{s=1}\frac{\theta_{s}^{4}}{\theta_{n+2}^{4}}  \right];
\\&\sum^{n+2}_{s,t=1}\left(\langle\nabla \mathcal{G}_{s\overline{t}},\nabla u_{i}\rangle^{2}
+\langle\nabla \mathcal{F}_{s\overline{t}},\nabla u_{i}\rangle^{2}\right)\leq\dfrac{2\overline{\theta}^{2}}{\theta^{2}_{n+2}}|\nabla u_{i}|^{2},
\end{aligned}\right.\end{equation}where

$$\overline{\theta}=\max_{1\leq s,t\leq n+2}\{\theta_{s\overline{t}}\}.$$
Hence, by \eqref{z-a}, if $\theta_{1}=\theta_{2}=\cdots=\theta_{n+2}=1$, we have

\begin{equation}\left\{\begin{aligned}\label{z-b}&\left[\sum^{n+2}_{s,t=1}\left(\nabla \mathcal{G}_{s\overline{t}}\cdot\nabla \mathcal{G}_{s\overline{t}}
+\nabla \mathcal{F}_{s\overline{t}}\cdot\nabla \mathcal{F}_{s\overline{t}}\right)\right]\Bigg{|}_{(1,1,\cdots,1)}=4n;
\\&
\left[\sum^{n+2}_{s,t=1}\left(\nabla \mathcal{G}_{s\overline{t}}\Delta \mathcal{G}_{s\overline{t}}+
\nabla \mathcal{F}_{s\overline{t}}\Delta \mathcal{F}_{s\overline{t}}\right)\right]\Bigg{|}_{(1,1,\cdots,1)}=0;
\\&\left[\sum^{n+2}_{s,t=1}\left(\Delta \mathcal{G}_{s\overline{t}}\Delta \mathcal{G}_{s\overline{t}}+\Delta
\mathcal{F}_{s\overline{t}}\Delta \mathcal{F}_{s\overline{t}}\right)\right]\Bigg{|}_{(1,1,\cdots,1)}=
16n\left(n + 1\right);
\\&\left[\sum^{n+2}_{s,t=1}\left(\langle\nabla \mathcal{G}_{s\overline{t}},\nabla u_{i}\rangle^{2}
+\langle\nabla \mathcal{F}_{s\overline{t}},\nabla u_{i}\rangle^{2}\right)\right]\Bigg{|}_{(1,1,\cdots,1)}=2|\nabla u_{i}|^{2}.
\end{aligned}\right.\end{equation}Therefore, it follows from \eqref{z-b} that

\begin{equation}\begin{aligned}\label{sum-pq-4}
\sum^{n+2}_{s,t=1}\left(\|u_{i}\nabla \mathcal{G}_{s\overline{t}}\|^{2}+\|u_{i}\nabla \mathcal{F}_{s\overline{t}}\|^{2}\right)&=
4n,
\end{aligned}\end{equation}and
\begin{equation}\begin{aligned}\label{sum-pq-5}
\sum^{n+2}_{s,t=1}\left(\|\langle\nabla \mathcal{G}_{s\overline{t}},\nabla u_{i}\rangle+u_{i}\Delta \mathcal{G}_{s\overline{t}}\|^{2}+\|\langle\nabla \mathcal{F}_{s\overline{t}},\nabla u_{i}\rangle+u_{i}\Delta \mathcal{F}_{s\overline{t}}\|^{2}\right)
\leq16n(n+1)+\dfrac{8\overline{\theta}^{2}}{\theta^{2}_{n+2}}\overline{\lambda}_{i}.
\end{aligned}\end{equation}
Recall that Cheng and Yang established the following general formula in \cite{CY2}:

\begin{equation}\label{c-y-formula}\sum^{k}_{i=0}(\overline{\lambda}_{k+1}-\overline{\lambda}_{i})^{2}\|h\nabla u_{i}\|^{2}\leq\sum^{k}_{i=0}
(\overline{\lambda}_{k+1}-\overline{\lambda}_{i})\|\langle\nabla h,\nabla u_{i}\rangle+u_{i}\Delta h\|^{2}.\end{equation}
Applying \eqref{c-y-formula} to the functions $\mathcal{G}_{s\overline{t}}(1,\cdots,1,Z^{1},\cdots,Z^{n+2})$
and $\mathcal{F}_{s\overline{t}}(1,\cdots,1,Z^{1},\cdots,Z^{n+2})$, we yield

\begin{equation*}\begin{aligned}
\sum^{k}_{i=1}\left(\overline{\lambda}_{k+1}-\overline{\lambda}_{i}\right)^{2}&\left(\|u_{i}\nabla \mathcal{G}_{s\overline{t}}\|^{2}
+\|u_{i}\nabla \mathcal{F}_{s\overline{t}}\|^{2}\right)\leq\\&
\left(\overline{\lambda}_{k+1}-\overline{\lambda}_{i}\right)\left(\|\langle\nabla \mathcal{G}_{s\overline{t}},\nabla u_{i}\rangle
+u_{i}\Delta \mathcal{G}_{s\overline{t}}\|^{2}+\|\langle\nabla \mathcal{F}_{s\overline{t}},\nabla u_{i}\rangle+u_{i}\Delta \mathcal{F}_{s\overline{t}}\|^{2}\right)
.
\end{aligned}\end{equation*} Taking sum on $s$ and $t$
from $1$ to $n+1$, one infer that

\begin{equation}\begin{aligned}\label{sum-pq-6}
\sum^{k}_{i=1}\left(\overline{\lambda}_{k+1}-\overline{\lambda}_{i}\right)^{2}\sum^{n+2}_{s,t=1}&\left(\|u_{i}\nabla \mathcal{G}_{s\overline{t}}\|^{2}+\|u_{i}\nabla \mathcal{F}_{s\overline{t}}\|^{2}\right)\leq\\&
\sum^{n+2}_{s,t=1}\left(\overline{\lambda}_{k+1}-\overline{\lambda}_{i}\right)\left(\|\langle\nabla \mathcal{G}_{s\overline{t}},\nabla u_{i}\rangle+u_{i}\Delta \mathcal{G}_{s\overline{t}}\|^{2}+\|\langle\nabla \mathcal{F}_{s\overline{t}},\nabla u_{i}\rangle+u_{i}\Delta \mathcal{F}_{s\overline{t}}\|^{2}\right)
.
\end{aligned}\end{equation}
Substituting \eqref{sum-pq-4} and \eqref{sum-pq-5} into \eqref{sum-pq-6}, we obtain

\begin{equation}\begin{aligned}\label{cy-universal-3}
\sum^{k}_{i=1}\left(\overline{\lambda}_{k+1}-\overline{\lambda}_{i}\right)^{2}\leq
\frac{2}{n}\sum^{k}_{i=1}\left(\overline{\lambda}_{k+1}-\overline{\lambda}_{i}\right)\left(2n(n+1)+\overline{\lambda}_{i}\right)
.
\end{aligned}\end{equation}Let  $\mu_{i}=\overline{\lambda_{i}}+2n(n+1)$. By \eqref{cy-universal-3} and \eqref{recursion}, we yield

\begin{equation}\begin{aligned}\label{cy-upper-1}
\overline{\lambda}_{k+1}+2n(n+1)\leq C_{0}(n)(\overline{\lambda}_{1}+2n(n+1))(k+1)^{\frac{1}{n}},
\end{aligned}\end{equation} where $C_{0}(n)$ is the same as the one in \eqref{Cheng-Yang-ineq}.
Therefore, we have
\begin{equation}\begin{aligned}\label{sum-pq}
\int_{M^{n}}\sum^{n+2}_{s,t=1}\left(\|u_{i}\nabla \mathcal{G}_{s\overline{t}}\|^{2}+\|u_{i}\nabla \mathcal{F}_{s\overline{t}}\|^{2}\right)dv
=16\sum^{n}_{s=1}\frac{\theta_{s}^{2}}{\theta_{n+2}^{2}}\left(\sum^{n}_{s=1}\frac{\theta_{s}^{2}}{\theta_{n+2}^{2}} + 1\right)+\dfrac{8\overline{\theta}^{2}}{\theta^{2}_{n+2}}\overline{\lambda}_{i}.
\end{aligned}\end{equation}  We choose $n+2$ positive real numbers $\theta_{s}$, such that, for all $s=1,2,\cdots,n+2$,
\begin{equation}\begin{aligned}\label{5.49} \sum_{s,t=1}^{n+1}\left[\int2 u_{i}\langle\nabla \mathcal{G}_{s\overline{t}},\nabla u_{i}\rangle\Delta  \mathcal{G}_{s\overline{t}}
dv+\int2 u_{i}\langle\nabla \mathcal{F}_{s\overline{t}},\nabla u_{i}\rangle\Delta  \mathcal{F}_{s\overline{t}}
dv\right]&=0,\end{aligned}\end{equation} and
\begin{equation*}a_{s\overline{t}}^{2}=\|\nabla \mathcal{G}_{s\overline{t}}u_{i}\|^{2}+\|\nabla \mathcal{F}_{s\overline{t}}u_{i}\|^{2}
\geq\sqrt{\||\nabla \mathcal{G}_{s\overline{t}}|^{2}u_{i}\|^{2}}+\sqrt{\||\nabla \mathcal{F}_{s\overline{t}}|^{2}u_{i}\|^{2}}=b_{s\overline{t}}\geq0.\end{equation*}
Let $$\beta=\min_{1\leq s,t\leq n+2}\{b_{s\overline{t}}\},$$
and $l=n+2$. Then, according to lemma \ref{lem2.1}, it follows from
\eqref{general-formula-2}, \eqref{sum-pq} and \eqref{5.49} that
\begin{equation*}\begin{aligned}\left[2\sum^{n}_{s=1}\frac{\theta_{s}^{2}}{\theta_{n+1}^{2}}+\frac{1}{2}(n+1)^{2}\beta\right]\left(\overline{\lambda}_{k+2}-\overline{\lambda}_{k+1}\right)^{2}
\leq4(\lambda_{k+2}+\rho)\left\{16\left[\left(\sum^{n}_{s=1}\frac{\theta_{s}^{2}}{\theta_{n+2}^{2}}\right)^{2}+\sum^{n}_{s=1}\frac{\theta_{s}^{4}}{\theta_{n+2}^{4}}\right]
+\dfrac{8\overline{\theta}^{2}}{\theta^{2}_{n+2}}\overline{\lambda}_{i}\right\},
\end{aligned}\end{equation*}for any $i,~i=1,2,\cdots,k$. Therefore, by the above inequality, we obtain

 \begin{equation}\begin{aligned}\label{po-gap}\overline{\lambda}_{k+2}-\overline{\lambda}_{k+1}
\leq\sqrt{\frac{4(\lambda_{k+2}+\rho)\left\{16\left[\left(\sum^{n}_{s=1}\frac{\theta_{s}^{2}}{\theta_{n+2}^{2}}\right)^{2}+\sum^{n}_{s=1}\frac{\theta_{s}^{4}}{\theta_{n+2}^{4}}\right]
+\dfrac{8\overline{\theta}^{2}}{\theta^{2}_{n+2}}\overline{\lambda}_{1}\right\}}
{2\sum^{n}_{s=1}\frac{\theta_{s}^{2}}{\theta_{n+1}^{2}}+\frac{1}{2}(n+1)^{2}\beta}}.
\end{aligned}\end{equation}
Furthermore, we put  $\rho=2n(n+1)$. Then, synthesizing \eqref{po-gap} and \eqref{cy-upper-1}, we obtain

\begin{equation*}\begin{aligned}\overline{\lambda}_{k+2}-\overline{\lambda}_{k+1}
&\leq\sqrt{\frac{4(\lambda_{k+2}+2n(n+1))\left\{16\left[\left(\sum^{n}_{s=1}\frac{\theta_{s}^{2}}{\theta_{n+2}^{2}}\right)^{2}+\sum^{n}_{s=1}\frac{\theta_{s}^{4}}{\theta_{n+2}^{4}}\right]
+\dfrac{8\overline{\theta}^{2}}{\theta^{2}_{n+2}}\overline{\lambda}_{1}\right\}}
{2\sum^{n}_{s=1}\frac{\theta_{s}^{2}}{\theta_{n+1}^{2}}+\frac{1}{2}(n+1)^{2}\beta}}\\
&\leq4\sqrt{\frac{4\left[\left(\sum^{n}_{s=1}\frac{\theta_{s}^{2}}{\theta_{n+2}^{2}}\right)^{2}+\sum^{n}_{s=1}\frac{\theta_{s}^{4}}{\theta_{n+2}^{4}}\right]
+\dfrac{2\overline{\theta}^{2}}{\theta^{2}_{n+2}}\overline{\lambda}_{1}}
{2\sum^{n}_{s=1}\frac{\theta_{s}^{2}}{\theta_{n+1}^{2}}+\frac{1}{2}(n+1)^{2}\beta}}\cdot\sqrt{C_{0}(n)(\overline{\lambda}_{1}+2n(n+1))}\cdot(k+1)^{\frac{1}{2n}}\\
&=C(n,M^{n})(k+1)^{\frac{1}{2n}},
\end{aligned}\end{equation*} where

\begin{equation*}C(n,M^{n})=\sqrt{\frac{4\left[\left(\sum^{n}_{s=1}\frac{\theta_{s}^{2}}{\theta_{n+2}^{2}}\right)^{2}+\sum^{n}_{s=1}\frac{\theta_{s}^{4}}{\theta_{n+2}^{4}}\right]
+\dfrac{2\overline{\theta}^{2}}{\theta^{2}_{n+2}}\overline{\lambda}_{1}}
{2\sum^{n}_{s=1}\frac{\theta_{s}^{2}}{\theta_{n+1}^{2}}+\frac{1}{2}(n+1)^{2}\beta}}\cdot\sqrt{C_{0}(n)(\overline{\lambda}_{1}+2n(n+1))}.\end{equation*}
Therefore, we finish the proof of this theorem.

\end{proof}
\begin{rem}In the proofs of theorem \ref{thm5.1} and theorem \ref{thm5.2}, the calculations of inequality \eqref{sys-5} and  inequality \eqref{z-a} is the same as in \cite{CY2}.\end{rem}

\section{Eigenvalues on Compact Homogeneous Riemannian Manifolds}\label{sec6}

In this section, we investigate the eigenvalue of the Laplacian on the compact homogeneous Riemannian manifolds. More specifically, we prove
the following theorem.
\begin{thm}\label{thm-homo}
Let $M^{n}$ be an $n$-dimensional compact homogeneous Riemannian
manifold without boundary.   If $\lambda_{i}$ , $ = 0, 1, 2,\cdots,$ is the $i$-th eigenvalue of
the closed eigenvalue problem \eqref{Eigen-Prob-closed}, then

\begin{equation*}\begin{aligned}\overline{\lambda}_{k+1}-\overline{\lambda}_{k}\leq
\overline{C}_{n,M^{n}}(k+1),\end{aligned}
\end{equation*}
where $$\overline{C}(n,M^{n})=4\sqrt{\frac{\overline{\alpha}^{2}\left[4\overline{\lambda}_{1}\sigma^{2}\overline{\lambda}_{1} +\overline{\lambda}^{2}_{1}
\sigma^{2}\right]}{\alpha\sigma^{2}\overline{\lambda}_{1}+d\beta}}\cdot\frac{\sqrt{5C_{0}(n)\overline{\lambda}_{1}}}{2}.$$\end{thm}

\begin{proof}
Recall that Cheng and Yang \cite{CY2} proved the following

\begin{equation}\label{c-y-univer}\sum^{k}_{i=0}(\overline{\lambda}_{k+1}-\overline{\lambda}_{i})^{2} \leq 4\sum^{k}_{i=0}
(\overline{\lambda}_{k+1}-\overline{\lambda}_{i})(\overline{\lambda}_{i}+\frac{1}{4}\overline{\lambda}_{1}).\end{equation}
By the \eqref{recursion} and \eqref{c-y-univer}, we obtian

\begin{equation}\label{c-y-upper}\overline{\lambda}_{k+1}\leq C_{0}(n)(k+1)^{2}\overline{\lambda}_{1}.\end{equation}
We assume that $\{g_{p}\}^{l}  _{p=1}$ is an orthonormal basis corresponding to the first eigenspace
 $E_{\overline{\lambda}_{1}}$ of the eigenvlaue problem

 \begin{equation*}\begin{aligned}
\Delta f =-\overline{\lambda}f, \ \ {\rm on}\ \ M.\end{aligned}\end{equation*}
It is equivalent to say that, the orthonormal basis $\{g_{p}\}^{l}_{p=1}$ satisfies the following

 \begin{equation}\begin{aligned}\label{eigen-first}\Delta g_{p} =-\overline{\lambda}_{1}g_{p},\ \ {\rm on}\ \  M.
\end{aligned}\end{equation}
Since  $M$ is an $n$-dimensional compact homogeneous Riemannian manifold without
boundary, we know that

 \begin{equation*}\begin{aligned}\sum^{l}_{p=1}g^{2}_{p}=\sigma^{2}
\end{aligned}\end{equation*} is constant (cf. Proposition 1 of Li
\cite{L}). Since the sum

 \begin{equation*}\begin{aligned}\sum^{l}_{p=1} g^{2}_{p}
=\sigma^{2}
\end{aligned}\end{equation*}   is a constant, we infer

 \begin{equation}\begin{aligned}&\label{g} \sum^{l}_{p=1}
g_{p}\nabla g_{p} = 0,\\&
\sum^{l}_{p=1}
\nabla g_{p}\cdot\nabla g_{p}=-
\sum^{l}_{p=1}
g_{p}\Delta g_{p} = \overline{\lambda}_{1}\sigma^{2}.
\end{aligned}\end{equation}
Hence, we infer from \eqref{g}

 \begin{equation}\begin{aligned}\label{g-1}\sum^{l}_{p=1}
\|u_{i}\nabla g_{p}\|^{2} =\overline{\lambda}_{1}\sigma^{2}.
\end{aligned}\end{equation}Let $\alpha_{1},\alpha_{2},\cdots,\alpha_{d}$ are $d$ positive numbers. We define $d$ scarling eigenfunctions $h_{p}(x)=\alpha_{p}g_{p}$, such that
\begin{equation}\begin{aligned}\label{su-0} \sum_{p=1}^{l}\int2 u_{i}\langle\nabla h_{p},\nabla u_{i}\rangle\Delta  h_{p}dv&=0,\end{aligned}\end{equation}
and
\begin{equation*}a_{p}^{2}=\|\nabla h_{p}u_{i}\|^{2}\geq\sqrt{\||\nabla h_{p}|^{2}u_{i}\|^{2}}=b_{p}\geq0,\end{equation*} where   $p=1,2,\cdots,d$.
Let $$\alpha=\min_{1\leq j\leq d}\{\alpha_{j}\},$$$$\overline{\alpha}=\max_{1\leq j\leq d}\{\alpha_{j}\},$$$$\beta=\min_{1\leq j\leq d}\{b_{j}\},$$
and $l=n+p$, then, by lemma \ref{lem2.3} and \eqref{g-1}, we have

\begin{equation}\begin{aligned}\label{a-b-20}\sum_{j=1}^{l}\frac{a_{j}^{2}+b_{j}}{2}
&=\sum^{l}_{p=1}
\frac{\|u_{i}\nabla(\alpha_{p} g_{p})\|^{2} +\sqrt{\||\nabla (\alpha_{p}g_{p})|^{2}u_{i}\|^{2}}}{2}\\&
\geq\frac{\alpha\sigma^{2}\overline{\lambda}_{1}+d\beta}{2}.
\end{aligned}\end{equation}
Furthermore, by \eqref{eigen-first}, one can deduce that

 \begin{equation}\begin{aligned}\label{(2.30)}\sum_{p=1}^{d}\|2\langle\nabla h_{p},\nabla u_{i}\rangle+u_{i}\Delta h_{p}\|^{2}&
=\sum_{p=1}^{d}\|2\langle\nabla (\alpha_{p}g_{p}),\nabla u_{i}\rangle+u_{i}\Delta (\alpha_{p}g_{p})\|^{2}\\
& =
\sum^{d}_{p=1}\int
_{M}
\left\{4\alpha_{p}^{2}(\nabla g_{p}\cdot\nabla u_{i})^{2}-4\overline{\lambda}_{1}g_{p}u_{i}\alpha_{p}^{2}
\nabla g_{p}\cdot\nabla u_{i} +\overline{\lambda}^{2}_{1}\alpha_{p}^{2}g^{2}_{p}u^{2}_{i}\right\}dv\\
&\leq2\sum^{d}_{p=1}\int
_{M}
\left\{4\alpha_{p}^{2}(\nabla g_{p}\cdot\nabla u_{i})^{2} +\overline{\lambda}^{2}_{1}\alpha_{p}^{2}g^{2}_{p}u^{2}_{i}\right\}dv\\
&\leq2\overline{\alpha}^{2}\left[\int_{M}
\sum^{d}_{p=1}
4(\nabla g_{p} \cdot\nabla u_{i})^{2}dv +\overline{\lambda}^{2}_{1}\sigma^{2}\right].
\end{aligned}\end{equation}
Since

 \begin{equation*}\begin{aligned}\langle\nabla g_{p},\nabla u_{i}\rangle^{2}\leq|\nabla g_{p}|^{2}|\nabla u_{i}|^{2}
\end{aligned}\end{equation*}  and

\begin{equation}\begin{aligned}\label{l-gp}\sum^{d}_{p=1}
|\nabla g_{p}|^{2} =\overline{\lambda}_{1}\sigma^{2},
\end{aligned}\end{equation}
we infer
from \eqref{su-0}, \eqref{(2.30)} and \eqref{l-gp} that

\begin{equation}\begin{aligned}\label{(2.31)}\sum_{p=1}^{d}\|2\langle\nabla h_{p},\nabla u_{i}\rangle+u_{i}\Delta h_{p}\|^{2}&\leq2\overline{\alpha}^{2}\left[\int_{M}
4\overline{\lambda}_{1}\sigma^{2}|\nabla u_{i} |^{2}dv+\overline{\lambda}^{2}_{1}
\sigma^{2}\right]
\\&= 2\overline{\alpha}^{2}\left[4\overline{\lambda}_{1}\sigma^{2}\overline{\lambda}_{i} +\overline{\lambda}^{2}_{1}
\sigma^{2}\right].
\end{aligned}\end{equation}
Since $\overline{\lambda}_{1}\alpha^{2}$ is positive, i.e., $\overline{\lambda}_{1}\alpha^{2} > 0,$ by making use of \eqref{general-formula-4}, \eqref{a-b-20}, \eqref{l-gp} and \eqref{(2.31)}, we have

\begin{equation}\begin{aligned}\label{ga1}\frac{\alpha\sigma^{2}\overline{\lambda}_{1}+d\beta}{2}\left(\overline{\lambda}_{k+2}-\overline{\lambda}_{k+1}\right)^{2}
\leq8(\overline{\lambda}_{k+2}+\rho)\overline{\alpha}^{2}\left[4\overline{\lambda}_{1}\sigma^{2}\overline{\lambda}_{i} +\overline{\lambda}^{2}_{1}
\sigma^{2}\right],
\end{aligned}\end{equation}for any $i=1,2,\cdots, k.$ Let $$\rho=\frac{1}{4}\overline{\lambda}_{1}.$$ Then, by using \eqref{c-y-upper} and \eqref{ga1}, one can infer that

\begin{equation}\begin{aligned}\overline{\lambda}_{k+2}-\overline{\lambda}_{k+1}&
\leq4\sqrt{\frac{\left(\overline{\lambda}_{k+2}+\frac{1}{4}\overline{\lambda}_{1}\right)\overline{\alpha}^{2}\left[4\overline{\lambda}_{1}\sigma^{2}\overline{\lambda}_{1} +\overline{\lambda}^{2}_{1}
\sigma^{2}\right]}{\alpha\sigma^{2}\overline{\lambda}_{1}+d\beta}}\\&
\leq4\sqrt{\frac{\overline{\alpha}^{2}\left[4\overline{\lambda}_{1}\sigma^{2}\overline{\lambda}_{1} +\overline{\lambda}^{2}_{1}
\sigma^{2}\right]}{\alpha\sigma^{2}\overline{\lambda}_{1}+d\beta}}\cdot\frac{\sqrt{5C_{0}(n)\overline{\lambda}_{1}}}{2}\cdot(k+2)\\&=
C(n,M^{n})(k+2),
\end{aligned}\end{equation}where $$C(n,M^{n})=4\sqrt{\frac{\overline{\alpha}^{2}\left[4\overline{\lambda}_{1}\sigma^{2}\overline{\lambda}_{1} +\overline{\lambda}^{2}_{1}
\sigma^{2}\right]}{\alpha\sigma^{2}\overline{\lambda}_{1}+d\beta}}\cdot\frac{\sqrt{5C_{0}(n)\overline{\lambda}_{1}}}{2}.$$
Thus, we finish the proof of this theorem.
\end{proof}

\section{Gap Coefficients}\label{sec7}
In theorem \ref{thm1.1}, the best constant $C_{n,\Omega}$ is called the gap coefficient. In this section,
we pay  attention to investigating the gap coefficient $C_{n,\Omega}$. It is worth noting
that it is very difficult for us to give the explicit expression of the optimal gap coefficient, even if $\Omega$ are some special domains in
the Euclidean space with  dimension $n$. However, we find that the eigenvalues depend on the shape of the bounded
domain $\Omega\subset\mathbb{R}^{n}$. Therefore, we introduce two new notations which will play significant roles in the estimating for the eigenvalues.

\begin{defn}Assume that $\Sigma_{1}$ and $\Sigma_{2}$ are two cubes  in $\mathbb{R}^{n}$, where $n\geq2,$ such that $\Sigma_{1}\subset\Omega\subset\Sigma_{2}$. We define

\begin{equation*}\mathfrak{S}_{1}(\Omega)=
\left\{ \begin{aligned}  \sup_{\Sigma_{1}\subset\mathbb{R}^{n},\Sigma_{2}\subset\mathbb{R}^{n}}\frac{d_{1}^{2}}{d_{2}^{2}},\ \ &n\geq2,
         \\
1,\ \ &n=1,
                          \end{aligned} \right.
                          \end{equation*}
and call it the first shape coefficient,  where $d_{1}$ denotes the length of the side of the cube $\Sigma_{1}$ and $d_{2}$ denotes the length of the side of the cube $\Sigma_{2}$, respectively.
Assume that $\mathbb{B}_{1}$ and $\mathbb{B}_{2}$ are two balls  in $\mathbb{R}^{n}$ such that $\mathbb{B}_{1}\subset\Omega\subset\mathbb{B}_{2}$. We define

\begin{equation*}\mathfrak{S}_{2}(\Omega)=
\left\{ \begin{aligned}  \sup_{\mathbb{B}_{1}\subset\mathbb{R}^{n},\mathbb{B}_{2}\subset\mathbb{R}^{n}}\frac{r_{1}^{2}}{r_{2}^{2}},\ \ &n\geq2,
         \\
1,\ \ &n=1,
                          \end{aligned} \right.
\end{equation*} and call it the second shape coefficient,  where $r_{1}$ denotes diameter  of the ball $\mathbb{B}_{1}$ and $r_{2}$ denotes diameter  of the ball $\mathbb{B}_{2}$, respectively.
\end{defn}
\begin{rem}When $\Omega$ is a cube then, $\Sigma_{1}=\Sigma_{2}$, therefore, we have
$\mathfrak{S}_{1}(\Omega)=1$. Similarly, when $\Omega$ is a ball, then $\mathbb{B}_{1}=\mathbb{B}_{2}$, therefore, we have
$\mathfrak{S}_{2}(\Omega)=1$.\end{rem}
According to a great amount of  numeric calculation for some special examples,  we venture to propose the following:
\begin{con}\label{gap-conj-z}
Let $\Omega$ be a bounded domain with piecewise smooth boundary $\partial\Omega$ on an $n$-dimensional Euclidean space
$\mathbb{R}^{n}$.  If $\lambda_{i}$ is the $i$-th eigenvalue of Dirichlet problem \eqref{Eigenvalue-Problem}, then, for any positive integer $k$,

\begin{equation}\label{con-z-1} \lambda_{k+1}-\lambda_{k}\leq \mathfrak{S}_{1}(\Omega)(\lambda_{2}-\lambda_{1})k^{\frac{1}{n}}.
\end{equation}
\end{con}
\begin{con}\label{gap-conj-z}
Let $\Omega$ be a bounded domain with piecewise smooth boundary $\partial\Omega$ on an $n$-dimensional Euclidean space
$\mathbb{R}^{n}$.  If $\lambda_{i}$ is the $i$-th eigenvalue of Dirichlet problem \eqref{Eigenvalue-Problem}, then, for any positive integer $k$,

\begin{equation}\label{con-z-1'} \lambda_{k+1}-\lambda_{k}\leq \mathfrak{S}_{2}(\Omega)(\lambda_{2}-\lambda_{1})k^{\frac{1}{n}}.
\end{equation}
\end{con}

\begin{rem} As we know, for the Dirichlet problem \eqref{Eigenvalue-Problem} on the $n$-dimensional Euclidean space $R^{n}$, the gap of the consecutive eigenvalues $\lambda_{k+1}-\lambda_{k}$ is
bounded by the first $k$-th eigenvalues in the previous literatures. However, from the above conjecture, we know that the gap of the consecutive eigenvalues
is bounded only by the first two eigenvalues.
\end{rem}

To exploit the gap coefficients, we discuss some important examples in the Euclidean space $\mathbb{R}^{n}$. We note that
there maybe exist more examples in the complete Riemannian manifolds to be found to suppose the conjecture.

\begin{exa} \label{exa1}the interval $(0,L)$
\end{exa}

When the dimension is one, Dirichlet problem \eqref{Eigenvalue-Problem} reads:
\begin{equation}
\left\{ \begin{aligned} \label{Eigen-Prob-1-d}\Delta u=-\lambda u,\ \ &{\rm
in}\ \ \ \ [0,L],
         \\
u=0,\ \ &{\rm on}\ \ \{0,L\}.
                          \end{aligned} \right.
                          \end{equation}
It is not difficult to infer that $$\lambda_{k+1}-\lambda_{k}=\frac{((k+1)\pi)^{2}}{L^{2}}-\frac{(k\pi)^{2}}{L^{2}}=(2k+1)\frac{\pi^{2}}{L^{2}}\leq\frac{3\pi^{2}}{L^{2}}k=(\lambda_{1}-\lambda_{2})k.$$
which implies that the  conjecture \ref{con-z-1} is true when the dimension $n=1$.

\begin{exa} The cuboid in $\mathbb{R}^{n}$
\end{exa}
Assume that $n\geq2,$ and $\Sigma_{0}(\subset\Omega\subset\mathbb{R}^{n})$ is a cuboid satisfying

$$V(\Sigma_{0})=\sup_{\Sigma\subset\Omega}V(\Sigma).$$ We define the gap coefficient as follows:

\begin{equation*}\mathcal{S}_{1}(\Omega)=
 \lambda_{2}(\Sigma_{0})-\lambda_{1}(\Sigma_{0}),
                          \end{equation*}where $\lambda_{2}(\Sigma_{0})$ and $\lambda_{1}(\Sigma_{0})$ are the first eigenvalue and the
second eigenvalue of the Dirichlet problem \eqref{Eigenvalue-Problem} of Laplacian on the cube $\Sigma_{0}\subset\mathbb{R}^{2}$, respectively.
 Under the above assumptions, we present the following:

\begin{con}\label{conj-z1}
Let $\Omega$ be a bounded domain with piecewise smooth boundary $\partial\Omega$ on an $n$-dimensional Euclidean space
$\mathbb{R}^{n}$. If $\lambda_{i}$ is the $i$-th eigenvalue of Dirichlet problem \eqref{Eigenvalue-Problem}, then

\begin{equation}\label{con-z1}\lambda_{k+1}-\lambda_{k}\leq\mathcal{S}_{1}(\Omega)k^{\frac{1}{n}}.
\end{equation}

\end{con}

\begin{rem}
Suppose that $\Omega$ is an arbitrary cuboid, and $ \Omega_{\star} $ is cube with the same volume as $\Omega$, i.e.,
$V(\Omega)=V(\Omega_{\star})$. Let $R_{\star}$ be the inscribe radii of the cube $\Omega_{\star}$. Assume that $\lambda_{1}(\Omega_{\star})$ and $\lambda_{2}(\Omega_{\star})$ are the first eigenvalue and the
second eigenvalue of the Dirichlet problem \eqref{Eigenvalue-Problem} of Laplacian on the cuboid $\Omega_{\star}\subset\mathbb{R}^{2}$, respectively. Then, it is easy to check that

\begin{equation}\label{gap-tran1}\lambda_{2}(\Omega_{\star})-\lambda_{1}(\Omega_{\star})=\frac{3n\pi^{2}}{D^{2}(\Omega_{\star})}.\end{equation}
Then, we have the following eigenvalue inequality (see {\rm\cite{Siu2}}):

\begin{equation}\label{siu-ineq1} (\lambda_{2}-\lambda_{1})R^{2}_{0}\leq(\lambda_{2}(\Omega_{\star})-\lambda_{1}(\Omega_{\star}))R_{\star}^{2},\end{equation}where $R_{0}$ denotes
the inradius of $\Omega$. Assume that the eigenvalues of the Dirichlet problem \eqref{Eigenvalue-Problem} of Laplacian on the cuboid $\Omega\subset\mathbb{R}^{2}$ satisfying \eqref{con-z-1}.
Then, from \eqref{gap-tran1} and \eqref{siu-ineq1}, we have

\begin{equation}\begin{aligned}  \lambda_{k+1}-\lambda_{k}&\leq(\lambda_{2}(\Omega_{\star})-\lambda_{1}(\Omega_{\star}))\frac{R_{\star}^{2}}{R^{2}_{0}}k^{\frac{1}{n}}
\\&=\frac{3n\pi^{2}}{D^{2}(\Omega_{\star})}\frac{R_{\star}^{2}}{R^{2}_{0}}k^{\frac{1}{n}}.\end{aligned}\end{equation}
\end{rem}

We assume that  $\Omega$ is the open $n$-dimensional rectangle $\Omega= (0,a_{1})\times\cdots\times (0,a_{n})\subset \mathbb{R}^{n},$
then, for the Dirichlet  eigenvalue problem on $\Omega$, the eigenvalues are given
by the collection $\left\{\lambda_{k_{1}\cdots k_{n}}\right\}$, where

$$\lambda_{k_{1}\cdots k_{n}}=\left(\frac{k^{2}_{1}}{a_{1}^{2}}+\cdots+\frac{k^{2}_{n}}{a_{n}^{2}}\right)\pi^{2}$$
and each $k_{j},j=1,\cdots,n,$ ranges over the positive integers.  For any fixed value of $a_{i},i=1,2,\cdots,n$, we can arrange all of the eigenvalues in order of size such that

$$\left\{\lambda_{i}\right\}_{i=1}^{+\infty}=\left\{\lambda_{k_{1}\cdots k_{n}}|k_{j}\in\mathbb{N}^{+},j=1,2,\cdots,n\right\}.$$
Assume  that $a_{1}=a_{2}=\cdots=a_{n}$, $k\leq100$. By a direct calculation, one can obtain the fundamental gap:

\begin{equation}\begin{aligned}\label{fun-gap-1}  \lambda_{2}-\lambda_{1}
=\frac{3n\pi^{2}}{D^{2}(\Omega)}.\end{aligned}\end{equation}Furthermore, by the numerical calculation, one can easily check that conjecture \ref{conj-z1} is true.
This is, by  \eqref{fun-gap-1}, we can prove the following

\begin{prop}\label{prop1}
Let $\Omega$ be a cube on the $n$-dimensional Euclidean space
$\mathbb{R}^{n}$. If $\lambda_{i}$ is the $i$-th eigenvalue of Dirichlet problem \eqref{Eigenvalue-Problem}, then, for any $k\leq100$,

\begin{equation}\lambda_{k+1}-\lambda_{k}\leq \frac{3\pi^{2}}{[d(\Omega)]^{2}}k^{\frac{1}{n}},
\end{equation}
where $d(\Omega)$ denotes the length of side of the cube $\Omega$.

\end{prop}

\begin{exa}The triangle in $\mathbb{R}^{2}$
\end{exa}
Assume that $\Omega$ is a  triangle on the plane $\mathbb{R}^{2}$, many mathematicians investigated the bounds for the eigenvalues of the Dirichlet problem \eqref{Eigenvalue-Problem} of Laplacian on $\Omega$, for example, in \cite{FS,He,LR,M,PS,Siu1,Siu2}. In particular,
for any triangle $\Omega\subset\mathbb{R}^{2}$ with diameter $D(\Omega)$,  Lu and  Rowlett \cite{LR} obtained a sharp lower bound of the fundamental gap as follows:

$$\lambda_{2}-\lambda_{1}\geq \frac{64\pi^{2}}{9D^{2}(\Omega)},$$
where equality holds if and only if $\Omega$ is equilateral, which affirmatively answers to a conjecture proposed by  Antunes-Freitas in \cite{AF}.
Furthermore, we assume that $\Pi_{0}\subset\Omega\subset\mathbb{R}^{2}$ is an equilateral triangle satisfying

$$V(\Pi_{0})=\sup_{\Pi\subset\Omega}V(\Pi).$$ Define the gap coefficient as follows:

\begin{equation*}\mathcal{S}_{2}(\Omega)=\lambda_{2}(\Pi_{0})-\lambda_{1}(\Pi_{0})=
\frac{64\pi^{2}}{9D^{2}(\Pi_{0})},
                          \end{equation*}where $D(\Pi_{0})$ denotes the diameter of the domain $\Pi_{0}$. Under those assumptions, we similarly present the following:

\begin{con}\label{conj-z2}
Let $\Omega$ be a bounded domain with piecewise smooth boundary $\partial\Omega$ on an $n$-dimensional Euclidean space
$\mathbb{R}^{n}$. If $\lambda_{i}$ is the $i$-th eigenvalue of Dirichlet problem \eqref{Eigenvalue-Problem}, then

\begin{equation}\label{con-z5}\lambda_{k+1}-\lambda_{k}\leq\mathcal{S}_{2}(\Omega)\sqrt{k}.
\end{equation}
\end{con}

\begin{rem}
Assume that $\Omega$ is an arbitrary triangle, and $ \Omega_{\ast} $ is an equilateral triangle with the same volume as $\Omega$, i.e., $V(\Omega)=V(\Omega_{\ast})$.
Let $R_{\ast}$ be the inscribe radii of the equilateral triangle $\Omega_{\ast}$. Suppose that $\lambda_{1}(\Omega_{\ast})$ and $\lambda_{2}(\Omega_{\ast})$
are the first eigenvalue and the second eigenvalue of the Dirichlet problem \eqref{Eigenvalue-Problem} of Laplacian on the equilateral
triangle $\Omega_{\ast}\subset\mathbb{R}^{2}$, respectively. Then,
we have

\begin{equation}\label{gap-tran}\lambda_{2}(\Omega_{\ast})-\lambda_{1}(\Omega_{\ast})=\frac{64\pi^{2}}{9D^{2}(\Omega_{\ast})}.\end{equation}
In {\rm\cite{Siu2}}, B. Siudeja proved the following eigenvalue inequality:

\begin{equation}\label{siu-ineq} (\lambda_{2}-\lambda_{1})R^{2}_{0}\leq(\lambda_{2}(\Omega_{\ast})-\lambda_{1}(\Omega_{\ast}))R_{\ast}^{2},\end{equation}where $R_{0}$ denotes
the inradius of $\Omega$. Assume that the eigenvalues of the Dirichlet problem \eqref{Eigenvalue-Problem} of Lapacian on a
triangle $\Omega\subset\mathbb{R}^{2}$ satisfying \eqref{con-z5}.
Then, from \eqref{gap-tran} and \eqref{siu-ineq}, we have

\end{rem}
\begin{equation}\begin{aligned}  \lambda_{k+1}-\lambda_{k}&\leq(\lambda_{2}(\Omega_{\ast})-\lambda_{1}(\Omega_{\ast}))\frac{R_{\ast}^{2}}{R^{2}_{0}}\sqrt{k}
\\&=\frac{64\pi^{2}}{9D^{2}(\Omega_{\ast})}\frac{R_{\ast}^{2}}{R^{2}_{0}}\sqrt{k}.\end{aligned}\end{equation}

Next, we suppose that  $\Omega$ is the open equilateral triangle in the $2$-dimensional Euclidean space $\mathbb{R}^{2}$,
then, for the Dirichlet  eigenvalue problem on $\Omega$, the eigenvalues are given
by the collection $\left\{\lambda_{mn}|m,n\in\mathbb{N}^{+}\right\}$, where

\begin{equation}\label{trang}\lambda_{m,n}=\frac{16\pi^{2}(m^{2} + mn + n^{2})}{9D^{2}(\Omega)},\end{equation}
and the positive integers $m$ and $n$ range  over the set of positive integer $\mathbb{Z}^{+}$.  According to the
size of the eigenvalues, these elements of the set $\{\lambda_{mn}\}$ can be put in increasing order such that

$$\left\{\lambda_{i}\right\}_{i=1}^{+\infty}=\left\{\lambda_{mn}|m,n\in\mathbb{N}^{+}\right\}.$$

It is easy to  see that the spectral structure of the equilateral triangle hinges upon the number
theoretic properties of the binary quadratic form $m^{2}+mn+n^{2}$. Therefore, from the point of
view of number theory, it is very difficult to obtain the estimates for the gap of the eigenvalues.
However, according to the numerical calculation, it is not difficult to check that conjecture \ref{conj-z2} is
true for any $k\leq100$, i.e.,, noticing \eqref{gap-tran}, one can prove the following:

\begin{prop} \label{prop2}
Assume that  $\Omega$ is an equilateral triangle on the plane $\mathbb{R}^{2}$, then the eigenvalues of the Dirichlet
problem \eqref{Eigenvalue-Problem} of Laplacian satisfy the inequality:

\begin{equation}\label{con-z1}\lambda_{k+1}-\lambda_{k}\leq\frac{64\pi^{2}}{9D^{2}(\Omega)}\sqrt{k},
\end{equation}
for any $k\leq100$. \end{prop}

\begin{rem}In fact, the order $k$ can be less than any finite positive integer in proposition \ref{prop1} and
proposition \ref{prop2}. However, if the condition $k\leq100$ is removed, then, from the point of view of number theory, we shall encounter an essential difficulty in the proofs of
proposition \ref{prop1} and proposition \ref{prop2}.\end{rem}

\begin{exa}the $n$-dimensional Ball $\mathbb{B}^{n}$ in $\mathbb{R}^{n}$
\end{exa}
Suppose that $\Omega_{\bullet}$ is an $n$-dimensional ball with the same volume as $\Omega$, i.e., $Vol(\Omega)=Vol(\Omega_{\bullet})$. Let $\lambda_{1}(\Omega_{\bullet})$ and $\lambda_{2}(\Omega_{\bullet})$ denote the first eigenvalue and the second eigenvalue of the Dirichlet problem of
Laplace operator on $\Omega_{\bullet}$, respectively. Recall that the famous Panye-P\'{o}lya-Weinberger conjecture (cf.\cite{Ash5,Ash6,PPW1,PPW2,T}) is to say
that, the ratios of the consecutive eigenvalues of Dirichlet  problem \eqref{Eigenvalue-Problem} satisfy the following

\begin{equation}\label{ppw-conj}\frac{\lambda_{k+1}}{\lambda_{k}}
\leq
\frac{\lambda_{2}(\Omega_{\bullet})}{\lambda_{1}(\Omega_{\bullet})}=\left(\frac{j_{n/2,1}}{j_{n/2-1,1}}\right)^{2},\end{equation}
where $j_{p,k}$ is the $k$-th positive zero of the Bessel
function $J_{p}(t)$.
In particular, when $k=1$, \eqref{ppw-conj} becomes
\begin{equation}\label{ppw-conj-1}\frac{\lambda_{2}}{\lambda_{1}}
\leq
\frac{\lambda_{2}(\Omega_{\bullet})}{\lambda_{1}(\Omega_{\bullet})}=\left(\frac{j_{n/2,1}}{j_{n/2-1,1}}\right)^{2},\end{equation}which is solved by Ashbaugh and Benguria(\cite{Ash2,Ash3,Ash4}).

\begin{rem}\label{gap-conj-z-2}
Let $\Omega$ be a bounded domain with piecewise smooth boundary $\partial\Omega$ on an $n$-dimensional Euclidean space
$\mathbb{R}^{n}$. Suppose that $\lambda_{i}$ is the $i$-th eigenvalue of Dirichlet problem \eqref{Eigenvalue-Problem} and satisfies

\begin{equation}\label{con-z-2} \lambda_{k+1}-\lambda_{k}\leq (\lambda_{2}-\lambda_{1})k^{\frac{1}{n}},
\end{equation}then, by \eqref{ppw-conj-1}, we can obtain the following estimate for the gap:
\begin{equation}\label{con-z-2} \lambda_{k+1}-\lambda_{k}\leq \lambda_{1}\left(\frac{j_{n/2,1}^{2}}{j_{n/2-1,1}^{2}}-1\right)k^{\frac{1}{n}}.
\end{equation}Therefore, eigenvalue inequality \eqref{con-z-2}  can be viewed as an algebraic inequality of  the Panye-P\'{o}lya-Weinberger type in the sense of the version of the gap of the consecutive eigenvalues. Obviously, it is a universal inequality.\end{rem}

Assume that  $R_{0}$ is the supremum of the radii among all of
the disks contained in $\Omega$   and $R_{\bullet}$ is the radii of the ball $\Omega_{\bullet}$, then we have (see {\rm\cite{Ash3,Siu2}})

\begin{equation*}\begin{aligned} \lambda_{k+1}-\lambda_{k}&\leq (\lambda_{2}(\Omega_{\bullet})-\lambda_{1}(\Omega_{\bullet}))\frac{R_{\bullet}^{2}}{R^{2}_{0}}k^{\frac{1}{n}},\end{aligned} \end{equation*}
which implies

\begin{equation}\begin{aligned} \lambda_{k+1}-\lambda_{k}\leq\frac{\left( j_{n/2,1}^{2}-j_{n/2-1,1}^{2}\right)}{R^{2}_{0}}k^{\frac{1}{n}}.\end{aligned} \end{equation}
In \cite{SWYY}, Singer-Wong-Yau-Yau obtained the following:

\begin{equation}\label{con-z-6} \lambda_{2}-\lambda_{1}\leq \frac{n\pi^{2}}{R^{2}_{0}}.
\end{equation}
Therefore, by utilizing \eqref{con-z-6}, we yield

\begin{equation*} \lambda_{k+1}-\lambda_{k}\leq \frac{n\pi^{2}}{R^{2}_{0}}k^{\frac{1}{n}}.
\end{equation*}

\vskip 5mm

\begin{ack} The author would like to thank Professor Zuoqin Wang  for his interest and helpful discussions. In particular, the author would like to  express his gratitude to professor Chiu-Yen Kao
and Braxton Osting for their useful comments and presenting an counterexample on the conjecture in the previous version of this paper. The author is supported by the National Nature Science Foundation of China (Grant No. 11401268).\end{ack}

\end{document}